\documentclass[12pt]{amsart}
\usepackage{amscd,amssymb,amsthm,verbatim}
\newcommand{\diag}{\operatorname{diag}}
\newcommand{\diam}{\operatorname{diam}}

\newcommand{\dvol}{\operatorname{dvol}}
\newcommand{\End}{\operatorname{End}}

\newcommand{\Id}{\operatorname{Id}}

\newcommand{\Ker}{\operatorname{Ker}}

\newcommand{\R}{{\mathbb R}}

\newcommand{\Rm}{\operatorname{Rm}}

\newcommand{\SL}{\operatorname{SL}}
\newcommand{\SO}{\operatorname{SO}}
\newcommand{\Sol}{\operatorname{Sol}}

\newcommand{\SU}{\operatorname{SU}}

\newcommand{\Tr}{\operatorname{Tr}}

\newcommand{\vol}{\operatorname{vol}}
\newcommand{\Z}{{\mathbb Z}}

\numberwithin{equation}{section}
\theoremstyle{plain}
\newtheorem{definition}[equation]{Definition}

\newtheorem{lemma}[equation]{Lemma}
\newtheorem{theorem}[equation]{Theorem}
\newtheorem{proposition}[equation]{Proposition}
\newtheorem{corollary}[equation]{Corollary}

\theoremstyle{remark}

\newtheorem{remark}[equation]{Remark}
\newtheorem{example}[equation]{Example}
\usepackage{graphicx}
\input{epsf}
\usepackage{a4wide}

\begin{document}

\title[On the initial geometry of a vacuum cosmological spacetime]
      {On the initial geometry of a vacuum cosmological spacetime}

\author{John Lott}
\address{Department of Mathematics\\
University of California, Berkeley\\
Berkeley, CA  94720-3840\\
USA} \email{lott@berkeley.edu}

\thanks{Research partially supported by NSF grant
DMS-1810700}
\date{February 16, 2020}

\begin{abstract}
  In the first part of this paper we consider expanding vacuum
  cosmological spacetimes with
  a free $T^N$-action. Among them, we give evidence that Gowdy spacetimes have
  AVTD (asymptotically velocity term dominated)
  behavior for their initial geometry,
  in any dimension.  We then give sufficient conditions to
  reach
  a similar conclusion about a
  $T^2$-invariant four dimensional
  nonGowdy spacetime. In the second part of the paper
  we consider vacuum cosmological spacetimes with crushing
  singularities. We introduce a monotonic
  quantity to characterize Kasner spacetimes.
Assuming scale-invariant  
  curvature bounds and local volume bounds,
  we give results about causal pasts. 
\end{abstract}

\maketitle


\section{Introduction} \label{sect1}

This paper is about the geometry of an expanding vacuum spacetime
that is diffeomorphic to $(0, t_0] \times X$, with $X$ compact, as one
  approaches the initial singularity at $t=0$.
  There are many open questions in this field, along with
  many partial results.
We refer to Isenberg's review \cite{Isenberg (2015)}.

The known results can be classified by how many
local symmetries are assumed.
Naturally, the more symmetries that are assumed, the stronger the
results.  Even in one extreme,
when spatial slices are locally homogeneous, the
asymptotic behavior is not completely understood.
It is also of interest to find any results in the other extreme, when one
assumes no local symmetries.

In this paper we only consider vacuum spacetimes.
Regarding the physical relevance of this restriction,
there are heuristic arguments that
under some assumptions, the matter content is not relevant for the
asymptotic behavior as one approaches an initial singularity
\cite[Chapter 4]{Belinski (2017)}. Suffice it to say that
results about vacuum spacetimes may have wider application.

  In Section \ref{sect2} we consider vacuum spacetimes with a free
  spatial $T^N$-action (possibly globally twisted) and a two dimensional
  quotient space. We first consider
  a Gowdy spacetime, meaning that
  the normal spaces to the orbits form an integrable distribution.
  Results about four dimensional
  Gowdy spacetimes are described in Ringstr\"om's review \cite{Ringstrom (2010)}.
In arbitrary dimension, the metric can be expressed in local coordinates by
\begin{equation}
  \overline{g} \: = \: \sum_{I,J=1}^N G_{IJ} \: dx^I \: dx^J \: + \:
  \sum_{\alpha, \beta = 1}^{2} g_{\alpha \beta} \: db^\alpha db^\beta.
\end{equation}
Here $\sum_{\alpha, \beta = 1}^{2} g_{\alpha \beta} \: db^\alpha db^\beta$
is a Lorentzian metric on the quotient space $B$. The matrix $(G_{IJ}) =
(G_{IJ})(b)$ is a $b$-dependent positive definite
symmetric $(N \times N)$-matrix.

As is standard, we assume that there is
an ``areal'' time coordinate $t \in (0, t_0]$ so that $\det(G) = t^N$,
c.f. \cite{Berger-Chrusciel-Isenberg-Moncrief (1997)}.
We write
$\sum_{\alpha, \beta = 1}^{2} g_{\alpha \beta} \: db^\alpha db^\beta
= - L^2 dt^2 + h dy^2$, where $y$ is a local coordinate on $S^1$.

One possible limiting behavior is
AVTD (asymptotically velocity term dominated) asymptotics.
With AVTD asymptotics,
as $t \rightarrow 0$, the leading asymptotics are given by the
VTD (velocity term dominated) equations, obtained by
dropping spatial derivatives in the evolution equations.
This is discussed in Sections 4-6 of Isenberg's review \cite{Isenberg (2015)}.

If we make a change of variable $t = e^{- \: \tau}$ then $\tau \rightarrow
\infty$ corresponds to approaching the singularity.
The VTD equation for $G$ is
\begin{equation} \label{0.2}
  (G^{-1} G_\tau)_\tau = 0.
  \end{equation}
By the choice of time parameter,
$\Tr \left( (G^{-1} G_\tau)_\tau \right) = (\ln \det G)_{\tau \tau} =
(- N \tau)_{\tau \tau} = 0$.  The content of
(\ref{0.2}) is that for each $y \in S^1$,
the normalized matrix
$(\det G)^{- \: \frac{1}{N}} G$ describes a geodesic, 
as a function of $\tau$, in the
symmetric space $\SL(N, \R)/\SO(N)$ of positive definite symmetric
$(N \times N)$-matrices with determinant one.
The AVTD
hypothesis for $G$ is that $(G^{-1} G_\tau)_\tau$ approaches zero as
$\tau \rightarrow \infty$.

In the case of four dimensional Gowdy spacetimes, i.e. when $N=2$,
Ringstr\"om
proved pointwise statements about the asymptotics of $G$, e.g. for
each $y \in S^1$
there is a limit $\lim_{\tau \rightarrow \infty}
(\det G(y, \tau))^{- \: \frac12}
G(y, \tau)$ 
in the ideal boundary of $H^2 = \SL(2, \R)/\SO(2)$, and the limit is
approached at an asymptotically constant speed
\cite{Ringstrom (2006)}. (This followed earlier work by
Isenberg and Moncrief on the polarized Gowdy case
\cite{Isenberg-Moncrief (1990)}.) One interesting feature is the
possible occurence of ``spikes'' in the spatial behavior as
$\tau \rightarrow \infty$
\cite{Berger-Garfinkle (1998),Berger-Moncrief (1993),Rendall-Weaver (2001)}.

We define an $H^{-1}$-Sobolev space of matrix-valued maps
on $S^1$ (equation (\ref{1.40})).
The following result roughly says that the $H^{-1}$-norm of
$(G^{-1} G_\tau)_\tau$ decays exponentially fast in $\tau$.

\begin{theorem} \label{0.3}
\begin{equation}
  \int_{\tau_0}^\infty e^{N\tau} \|
  (G^{-1} G_\tau)_\tau \|^2_{H^{-1}_{\tau}} \: d\tau < \infty.
\end{equation}
\end{theorem}

Hence there is AVTD-like behavior.
The appearance of the $H^{-1}$-Sobolev space is not unreasonable, in view of
the possible occurence of spikes in the spatial behavior.
Compared to earlier results, one difference
in Theorem \ref{0.3} is the use of the
Sobolev norm to measure the AVTD-like behavior. The norm arises from the use of
a monotonic functional, that in fact
differs in the nonpolarized case from those
previously considered.  Another feature is that the result is somewhat
more geometric, in that it holds in
arbitrary dimension.

We next consider four dimensional spacetimes that have a free spatial
$T^2$-action but are nonGowdy, where there are fewer results. 
As in the paper \cite{LeFloch-Smulevici (2015)} by
LeFloch and Smulevici, the metric has a local expression
\begin{equation}
  g = e^{2(\eta - U)} (-dR^2 + a^{-2} d\theta^2) +
  e^{2U} (dx^1 + Adx^2 + (G+AH) d\theta)^2 +
  e^{-2U}R^2(dx^2+Hd\theta)^2.
\end{equation}
Here the time parameter $R$ is such that the area of the $T^2$-orbit is
$R$.
The variables $\eta$, $U$, $a$, $A$, $G$ and $H$ are functions of $R$ and
$\theta$.
We make a change of variable $R = e^{- \tau}$.
The AVTD asymptotics for $U$ are that $a(a^{-1}U_\tau)_\tau
\: - \: \frac12 e^{2 \tau} e^{4U} A_\tau^2$ goes to zero as
$\tau \rightarrow \infty$
\cite{Ames-Beyer-Isenberg-LeFloch (2013),Clausen-Isenberg (2007)}.

Unlike in the Gowdy case, one does not expect AVTD-like behavior in
general. Some solutions with a ``half-polarized'' condition on $A$
were
constructed by Ames-Beyer-Isenberg-LeFloch using Fuchsian methods 
\cite{Ames-Beyer-Isenberg-LeFloch (2013)};
those solutions have AVTD-like behavior.
The next theorem gives a sufficient condition for
AVTD-like behavior to hold for $U$.

\begin{theorem} \label{0.6}
  If $\int_{S^1} H \: d\theta$ is bounded below as
  $\tau \rightarrow \infty$, and
\begin{equation} \label{0.7}
\int_{\tau_0}^\infty e^{2\tau}  \| e^{4U} a^2 A_\theta^2 \|_{H^{-1}_{\tau}}^2 \: d\tau < \infty,
\end{equation}
 then 
 \begin{equation}
   \int_{\tau_0}^\infty e^{2\tau}  \|
   a(a^{-1}U_\tau)_\tau
   \: - \: \frac12 e^{2 \tau} e^{4U} A_\tau^2
   \|_{H^{-1}_{\tau}}^2 \: d\tau < \infty.
\end{equation}
\end{theorem}

The expression $\int_{S^1} H \: d\theta$ is a holonomy-type term.
The condition (\ref{0.7}) is consistent with the results of
\cite{Ames-Beyer-Isenberg-LeFloch (2013)}, where
$A$ is half-polarized. In those solutions,
(\ref{0.7}) is satisfied.  When $A$ is not half-polarized, the
construction in \cite{Ames-Beyer-Isenberg-LeFloch (2013)}
breaks down. Numerics indicate that general
$T^2$-invariant nonGowdy solutions are not AVTD, and instead have
Mixmaster-type behavior \cite{Berger-Isenberg-Weaver (2001)}.
We do not have anything to say about Mixmaster dynamics, but the results of
the paper may help to clarify the line between AVTD dynamics and
Mixmaster dynamics.

The proofs of Theorems \ref{0.3} and \ref{0.6} involve
finding energy expressions that
are monotonically nondecreasing in real time, integrating the
derivative to get an integral bound on spatial derivative terms, and
then applying the evolution equation.

In Section \ref{sect3}
we consider vacuum spacetimes or, equivalently,
Einstein flows, without any assumed symmetries.
The spacetime is diffeomorphic to
$(0,T_0] \times X$, where $X$ is compact.  In this introduction we
  take $\dim(X) = 3$, although some of the results are true
  for general dimension.
  We assume that there is a crushing singularity as $t \rightarrow 0$,
  meaning that
there is a
sequence of compact Cauchy hypersurfaces going toward the end at
$\{0\} \times X$ whose mean curvatures approach $- \infty$ uniformly.
From Gerhardt's paper \cite{Gerhardt (1983)},
there is a foliation near the end by
constant mean curvature (CMC) compact spatial hypersurfaces, whose mean
curvatures $H$ approach $- \infty$.
We then take $t = - \frac{3}{H}$, the Hubble time,
which ranges in an interval $(0, t_0]$.
The spacetime metric can be written as
$g = - L^2 dt^2 + h(t)$, where $h(t)$ is a Riemannian metric on $X$.

Fischer and Moncrief showed that
the quantity $t^{-3} \vol(X, h(t))$ is monotonically
  nonincreasing in $t$, and is constant if and only if the spacetime is
  a Lorentzian cone over a hyperbolic $3$-manifold
  \cite{Fischer-Moncrief (2002)}. (A similar result was proven by Anderson
  \cite{Anderson (2001)}.) This had implications for the
  long-time behavior of expanding
  spacetimes that live instead on
  $[t_0, \infty)$, and gave rise to the intuition that
  most of such a spacetime, in the sense of volume, should approach
  such a Lorentzian cone; a precise statement is in
  \cite[Section 2.2]{Lott (2018)}. 
  In this paper we are concerned with the behavior in the
  shrinking direction,
  as $t \rightarrow 0$. It turns out that
  $t^{-1} \vol(X, h(t))$ is a partial analog to the
  Fischer-Moncrief quantity.

  \begin{theorem} \label{0.9}
    We have
\begin{equation} \label{0.10}
  \frac{d}{dt} \left( t^{-1} \vol(X, h(t)) \right) =
  - \: \frac{1}{3} \int_X L R \: \dvol_h.
  \end{equation}
Hence
  \begin{equation} 
    \int_0^{t_0} \int_X (- t^2 R) \: L \: \frac{\dvol_{h(t)}}{t}
      \: \frac{dt}{t} < \infty. 
  \end{equation}
\end{theorem}

  One sees from (\ref{0.10}) that $t^{-1} \vol(X, h(t))$ is monotonically
  nondecreasing in $t$ provided that the spatial scalar curvature $R$ is
  nonpositive.  The next result characterizes the equality case.

  \begin{theorem} \label{0.12}
  Suppose that
  $R \le 0$ and $t_1^{-1} \vol(X, h(t_1)) = t_2^{-1} \vol(X, h(t_2))$,
  for some $t_1 < t_2$.
  Suppose that
  $X$ is orientable and that
  there is an aspherical component in the prime decomposition of $X$.
  Then 
  the Einstein flow is a Kasner solution. 
  \end{theorem}

  There is a natural rescaling (\ref{old1.36}) of a CMC Einstein flow.
  Using Theorem \ref{0.9}, one can show that
  if $R \le 0$ then as one approaches the singularity, there is
  Kasner-like geometry in an integral sense, relative to a limiting
  measure. Namely, put
  $\dvol_0 = \lim_{t \rightarrow 0} t^{-1} \dvol_{h(t)}$; 
  this limit exists as a measure, although it may be zero.
  Let $K$ denote the second fundamental form of the spatial hypersurfaces.
  
  \begin{theorem} \label{0.13}
  Suppose that $R \le 0$. Given $\Lambda > 1$, we have
\begin{equation} 
  \lim_{s \rightarrow 0} \left| L_s - \frac{1}{3} \right| =
  \lim_{s \rightarrow 0}
  \left| |K_s|^2 - \frac{9}{u^2} \right| = \lim_{s \rightarrow 0} |R_s| = 0
\end{equation}
in $L^1 \left( [\Lambda^{-1}, \Lambda] \times X, du \dvol_0 \right)$.
\end{theorem}

  The analogy between $t^{-3} \vol(X, h(t))$ (for the expanding direction)
  and $t^{-1} \vol(X, h(t))$ (for the shrinking direction)
   is only partial.
  First, $t^{-1} \vol(X, h(t))$ is only monotonic when $R \le 0$. Second,
  $t^{-3} \vol(X, h(t))$ is invariant under rescaling,
  whereas $t^{-1} \vol(X, h(t))$ is not.

The remaing results of the paper involve a curvature assumption.
Let $|\Rm|_T$ denote the norm of the spacetime curvature, as given in
(\ref{old1.47}).
Following Ricci flow terminology, we define
a type-I Einstein flow to be a CMC Einstein flow
  for which there is some $C < \infty$ so that
  $|\Rm|_T \le C t^{-2}$ for all $t \in (0, t_0]$. We show that
  except for a clear counterexample, the normalized spatial diameters in a
  type-I Einstein flow go to infinity as $t \rightarrow 0$.

\begin{theorem} \label{0.15}
  Suppose that a type-I Einstein flow ${\mathcal E}$ satisfies
  $\liminf_{t \rightarrow 0} t^{-1} \diam(X, h(t)) < \infty$.
Then ${\mathcal E}$ is a
  Lorentzian cone over 
  a compact hyperbolic $3$-manifold.
\end{theorem}

Let $B_{h(t)}(x,t)$ denote the time-$t$ spatial
metric ball of radius $t$ around
$x \in X$.
We say that a CMC Einstein flow ${\mathcal E}$ is {\em noncollapsed}
    if there is some $v_0 > 0$ so that for all $(t,x) \in
    (0, t_0] \times X$, we have
      $\vol \left( B_{h(t)}(x, t)  \right) \ge v_0 t^3$.
      Since we have mentioned the two dichotomies
      shrinking/expanding and collapsed/noncollapsed,
      let us clarify the difference.  As $H$ is negative,
      we are considering flows for which the volume of the time-$t$ slice is
      shrinking as $t \rightarrow 0$ and
      expanding as $t \rightarrow \infty$. In contrast, the notion of
      collapsed/noncollapsed is based on the {\em normalized} 
      volumes of
      metric balls in the time slices.  There are many examples of
      Einstein flows that are collapsed in the expanding direction, as
      discussed in \cite{Lott (2018)}. In contrast, Einstein flows tend to
      be noncollapsed in the shrinking direction.

      In this paper we focus on noncollapsed type-I Einstein flows.
      The motivation comes from
      looking at examples of 
      crushing singularities. There may be crushing
      singularities that are not type-I, or are type-I but collapsed.
      If there are such examples then the methods of 
      \cite[Sections 3 and 4]{Lott (2018)}
      would give some information about them.

      Noncollapsed type-I Einstein flows have the technical advantage that
      one can take rescaling limits.
      In view of the BKL conjectures
      \cite{BKL (1970),Belinski (2017)},
      the possible existence of particle horizons is relevant for
      understanding initial
      singularities. One question is whether there are
      distinct points $x_1, x_2 \in X$
      so that for $t$ sufficiently small, the causal pasts
      $J_-(x_1, t)$ and $J_-(x_2, t)$ are disjoint. In general, this need not
      be the case.  However, we show that except for a clear counterexample,
      if $t$ is small enough then there are many points whose causal pasts are
      mutually disjoint on a relatively long backward time interval. 

\begin{theorem} \label{0.16}
  Let ${\mathcal E}$ be a noncollapsed type-I CMC Einstein flow.
Then either
  \begin{enumerate}
      \item ${\mathcal E}$ is a
  Lorentzian cone over 
  a compact hyperbolic $3$-manifold, or
    \item Given
  $N \in \Z^+$, $\Lambda > 1$ and $x^\prime \in X$, there is some $\widehat{t} \in
  (0, t_0]$ with the following property.  Given $t \in (0, \widehat{t}]$,
      there are $N$ points $\{x_j\}_{j=1}^N$ in $X$, with $x_1 = x^\prime$, so that if $j \neq
      j^\prime$ then the causal pasts $J^-(x_j,t)$ and $J^-(x_{j^\prime},t)$
      are disjoint on the time interval $[\Lambda^{-1} t, t]$.
      \end{enumerate}
\end{theorem}

One can localize the preceding result to an
arbitrary open subset of $X$.

\begin{theorem} \label{0.17}
  Let ${\mathcal E}$ be a noncollapsed type-I CMC Einstein flow.
  Given $N \in \Z^+$, $\Lambda > 1$, $\epsilon > 0$, $\alpha \in (0,1)$,
  an open set $U \subset X$ and a point $x^\prime \in U$, there is
  some $\widehat{t} \in (0, t_0]$ with the following property. For $t \in (0, \widehat{t}]$, either
    \begin{enumerate}
    \item     The rescaled pointed flow ${\mathcal E}_{t}$ on $(X, x^\prime)$ is $\epsilon$-close in the
      pointed $C^{1,\alpha}$-topology to a Lorentzian cone 
      over a region in a hyperbolic $3$-manifold, having $U$
      as a bounded subset of the approximation region, or
 \item There are $N$ points $\{x_j\}_{=1}^N$ in $U$, with $x_1 = x^\prime$, so that if $j \neq
      j^\prime$ then the causal pasts $J^-(x_j,t)$ and $J^-(x_{j^\prime},t)$
      are disjoint on the time interval $[\Lambda^{-1} t, t]$.      
      \end{enumerate}
\end{theorem}

There is also a measure theoretic version 
(Proposition \ref{2.71}).

The structure of the paper is the following. In Section \ref{sect2} we
prove
Theorems \ref{0.3} and \ref{0.6}.  In Section \ref{sect3} we prove the
remaining theorems.
More detailed descriptions are at the
beginnings of the sections.

I thank the referees for helpful comments.

\section{Torus symmetries} \label{sect2}

In this section we prove the results about $T^N$-actions.
In Subsection \ref{subsect2.1} we recall results about the geometry of
spacetimes with free isometric $T^N$-actions
 (possibly globally twisted).  In Subsection
\ref{subsect2.3} we prove Theorem \ref{0.3} and in Subsection
\ref{subsect2.4} we prove Theorem \ref{0.6}.

\subsection{Geometric setup} \label{subsect2.1}

We begin with the geometric setup of
\cite[Section 4.1]{Lott (2010)}, to which we refer for more details.
Let ${\mathcal G}$ be an $N$-dimensional abelian Lie group, with Lie algebra
${\frak g}$.
Let ${\frak E}$
be a local system on $B$ of Lie groups isomorphic to ${\mathcal G}$.
There is a corresponding flat ${\frak g}$-vector bundle $e$ on $B$;
see \cite[Section 4.1]{Lott (2010)}.

Let $M$ be the total space of an
${\frak E}$-twisted principal ${\mathcal G}$-bundle
with base $B$, in the sense of \cite[Section 4.1]{Lott (2010)}.
(An example is when ${\frak E}$ is the constant local system and
$M$ is the total space of a $T^N$-bundle on $B$.)
We write $\dim(B) = n+1$ and $\dim(M) = m = N+n+1$.

Let
$\overline{g}$ be a Lorentzian metric on $M$ with a
free local isometric ${\frak E}$-action. We assume that the
induced metrics on the ${\frak E}$-orbits are Riemannian.
In adapted coordinates, we can write
\begin{equation} \label{A.1}
  \overline{g} \: = \: \sum_{I,J=1}^N G_{IJ} \: (dx^I + A^I) (dx^J + A^J) \: + \:
  \sum_{\alpha, \beta = 1}^{n+1} g_{\alpha \beta} \: db^\alpha db^\beta.
\end{equation}
Here
$G_{IJ}$ is the local expression of a Euclidean inner product on
$e$,
$\sum_{\alpha, \beta = 1}^{n+1} g_{\alpha \beta} \: db^\alpha db^\beta$ is
the local expression of a Lorentzian metric $g_B$ on $B$ and
$A^I = \sum_{\alpha} A^I_\alpha db^\alpha$ are the components of
a local $e$-valued $1$-form describing a connection $A$ on the
twisted ${\frak G}$-bundle $M \rightarrow B$.

Put $F^I_{\alpha \beta} = \partial_\alpha A^I_\beta -
\partial_\beta A^I_\alpha$.
At a given point $b \in B$, we can assume that $A^I(b) = 0$.
We write
\begin{equation} \label{A.2}
  G_{IJ;\alpha \beta} \: = \: G_{IJ,\alpha \beta} \: - \:
  \Gamma^{\sigma}_{\: \: \alpha \beta} \: G_{IJ, \sigma},
\end{equation}
where $\{\Gamma^{\sigma}_{\: \: \alpha \beta}\}$ are the
Christoffel symbols for the metric $g_{\alpha \beta}$ on $B$.

From \cite[Section 4.2]{Lott (2010)},
the Ricci tensor of ${\overline g}$ on $M$
is given in terms of the curvature tensor
$R_{\alpha \beta \gamma \delta}$ of $B$, the $2$-forms $F^I_{\alpha \beta}$
and the metrics $G_{IJ}$ by
\begin{align} \label{A.3}
  \overline{R}_{IJ}^{\overline{g}}
  \:  =  \: & - \: \frac12 \: g^{\alpha \beta} \:
  G_{IJ; \alpha \beta} \: - \: \frac14 \: g^{\alpha \beta} \:
  G^{KL} \: G_{KL, \alpha} \: G_{IJ, \beta} \: + \:
  \frac12 \: g^{\alpha \beta} \: G^{KL} \: G_{IK, \alpha} \:
  G_{LJ, \beta} \: + \\
  &  \frac14 \: g^{\alpha \gamma} \: g^{\beta \delta} \:
  G_{IK} \: G_{JL} \: F^K_{\alpha \beta} \: F^L_{\gamma \delta} \notag \\
  \overline{R}_{I \alpha}^{\overline{g}}
  \:  =  \: & \frac12 \: g^{\gamma \delta} \:
  G_{IK} \: F^K_{\alpha \gamma; \delta} \: + \: \frac12 \: g^{\gamma \delta} \:
  G_{IK, \gamma} \: F^K_{\alpha \delta} \: + \: \frac14 \:
  g^{\gamma \delta} \: G_{Im} \: G^{KL} \: G_{KL, \gamma} \: F^m_{\alpha \delta}
  \notag \\
  \overline{R}_{\alpha \beta}^{\overline{g}}
  \:  =  \: & R_{\alpha \beta}^g \: - \:
  \frac12 \: G^{IJ} \: G_{IJ; \alpha \beta} \: + \: \frac14 \:
  G^{IJ} \: G_{JK,\alpha} \: G^{KL} \: G_{LI,\beta} \: - \:
  \frac12 \: g^{\gamma \delta} \: \: G_{IJ} \: F^I_{\alpha \gamma} \:
  F^J_{\beta \delta}. \notag
\end{align}
The scalar curvature is
\begin{align} \label{A.4}
  \overline{R}^{\overline{g}} \: = \: &
  R^g \: - \: g^{\alpha \beta} G^{IJ}  \: G_{IJ; \alpha \beta} \: + \:
  \frac34 \: g^{\alpha \beta} \: G^{IJ} \: G_{JK, \alpha} \: G^{KL} \:
  G_{LI, \beta} \\
  & \: - \: \frac14 \: g^{\alpha \beta} \: G^{IJ} \:
  G_{IJ, \alpha} \: G^{KL} \: G_{KL, \beta} \: -
  \: \frac14 \:
  g^{\alpha \gamma} \: g^{\beta \delta} \: G_{IJ} \:
  F^I_{\alpha \beta} \: F^J_{\gamma \delta}. \notag
\end{align}

In what follows we will assume that the flat vector bundle $e$ has
holonomy in $\SL(N, \R)$, so that $\ln \det G$ is globally defined
on $B$. We have
\begin{equation} \label{A.5}
  \nabla_\alpha \ln \det G = G^{IJ} G_{IJ, \alpha}
\end{equation}
and
\begin{equation} \label{A.6}
  \triangle_g \ln \det G =
  g^{\alpha \beta} G^{IJ} G_{IJ; \alpha \beta} -
  g^{\alpha \beta} G^{IJ} G_{JK, \alpha} G^{KL} G_{LK, \beta}.
\end{equation}
Writing
\begin{equation} \label{A.7}
  |F|^2 = G_{IJ} g^{\alpha \beta} g^{\gamma \delta} F^I_{\alpha \gamma}
  F^J_{\beta \delta},
\end{equation}
the first equation in (\ref{A.3}) gives
\begin{equation} \label{A.8}
  G^{IJ} \overline{R}_{IJ} =
  - \frac12 \triangle_g \ln \det G - \frac14 g^{\alpha \beta}
  (\nabla_\alpha \ln \det G) (\nabla_\beta \ln \det G) + \frac14 |F|^2.
\end{equation}
Note that $|F|^2$ need not be
nonnegative.

Given a foliation of $B$ by
compact spacelike hypersurfaces $Y$, we can write the metric $g$ on $B$ as
\begin{equation} \label{A.9}
  g = - L^2 dt^2 + \sum_{i,j=1}^n h_{ij} dy^i dy^j.
\end{equation}
Here $L = L(y,t)$ is the lapse function and we have performed
spatial diffeomorphisms to kill the shift vectors.

Suppose hereafter that
$\det G$ is spatially constant, i.e. only depends on $t$
\cite{Berger-Chrusciel-Isenberg-Moncrief (1997)}.
Then
\begin{equation} \label{A.10}
  g^{\alpha \beta}
  (\nabla_\alpha \ln \det G) (\nabla_\beta \ln \det G) =
  - \: L^{-2} ( \partial_t \ln \det G )^2
\end{equation}
and
\begin{equation} \label{A.11}
  \triangle_g \ln \det G =
  - \: \frac{1}{L \sqrt{\det h}} \partial_t
  \left( L^{-1} \sqrt{\det h} (\partial_t \ln \det G) \right).
\end{equation}

If $\overline{R}^{\overline{g}}_{IJ} = 0$ then (\ref{A.8}) becomes
\begin{equation} \label{A.12}
  \partial_t \left( L^{-1} \sqrt{\det G} (\partial_t \ln \det G) \sqrt{\det h} \right) +
  \frac12 L \sqrt{\det G} |F|^2 \sqrt{\det h} = 0.
\end{equation}
More invariantly,
\begin{equation} \label{1.13} 
  \partial_t \left( L^{-1} \sqrt{\det G} (\partial_t \ln \det G) \dvol_{h(t)} \right)  = 
  - \: \frac12 L \sqrt{\det G} |F|^2 \dvol_{h(t)}.
\end{equation}
In particular, if
$F = 0$ then 
\begin{equation} \label{1.14}
\mu =   L^{-1} \sqrt{\det G} (\partial_t \ln \det G) \dvol_{h(t)}
  \end{equation}
is a $t$-independent smooth positive density on $Y$.

We suppose in the rest of this section that $\dim(B) = 2$, i.e. $\dim(Y) = 1$.
We write $g$ locally (in $Y$) as $- L^2 dt^2 + h dy^2$.

\subsection{Gowdy spacetime} \label{subsect2.3}

In this subsection we assume that $F=0$.
(This is automatic, for
example,
if $X$ is a three dimensional $\Sol$-manifold \cite[p. 2288]{Lott (2018)}.)
Let $\mu$ be the $t$-independent density on $Y$ defined in (\ref{1.14}).

Put
\begin{align} \label{A.15}
  {\mathcal E}(t) = & \int_{Y} \left[ h^{-1} \Tr \left( \left( G^{-1}
    \frac{\partial G}{\partial y} \right)^2 \right) + L^{-2}
    \Tr \left( \left( G^{-1}
    \frac{\partial G}{\partial t} \right)^2 \right) \right] L \dvol \\
  = & \int_{Y} \left[ L h^{- \: \frac12} \Tr \left( \left( G^{-1}
    \frac{\partial G}{\partial y} \right)^2 \right) + L^{-1} h^{\frac12}
    \Tr \left( \left( G^{-1}
    \frac{\partial G}{\partial t} \right)^2 \right) \right] dy. \notag
\end{align}
Still assuming that $F=0$,
if $\overline{R}^{\overline{g}}_{IJ} = 0$ then
equation (\ref{A.3}) gives the matrix equation
\begin{align} \label{A.17}
  & - L^{-2} (G^{-1} G_{tt} - G^{-1} G_t G^{-1} G_t) +
  h^{-1} (G^{-1} G_{yy} - G^{-1} G_y G^{-1} G_y) + \\
  & L^{-3} L_t G^{-1} G_t + L^{-1} h^{-1} L_y G^{-1} G_y
  - \frac12 L^{-2} h^{-1} h_t G^{-1} G_t - \notag \\
  & \frac12
  h^{-2} h_y G^{-1} G_y - \frac12 L^{-2} (\ln \det G)_t G^{-1} G_t = 0. \notag
\end{align}
Suppose that
$(\ln \det G)_t > 0$.
Using (\ref{A.17}) and the $t$-independence of $\mu$, one finds

\begin{align} \label{A.18}
  \frac{d{\mathcal E}}{dt}  = &
  \int_Y \left( \frac{(\ln \det G)_{tt}}{(\ln \det G)_{t}} + \frac12 (\ln \det G)_t \right) L h^{- \: \frac12}
  \Tr \left( \left( G^{-1} G_y \right)^2 \right) \: dy + \\
  & \int_Y \left( \frac{(\ln \det G)_{tt}}{(\ln \det G)_{t}} - \frac12 (\ln \det G)_t \right) L^{-1} h^{\frac12}
  \Tr \left( \left( G^{-1} G_t \right)^2 \right) \: dy \notag \\
  = &
  \int_Y \left( \frac{(\ln \det G)_{tt}}{(\ln \det G)_{t}} + \frac12 (\ln \det G)_t \right) L h^{- 1}
  \Tr \left( \left( G^{-1} G_y \right)^2 \right) \: \dvol + \notag \\
  & \int_Y \left( \frac{(\ln \det G)_{tt}}{(\ln \det G)_{t}} - \frac12 (\ln \det G)_t \right) L^{-1} 
  \Tr \left( \left( G^{-1} G_t \right)^2 \right) \: \dvol. \notag
\end{align}

A scale invariant quantity that is monotonically nonincreasing in $t$ is given by
\begin{equation} \label{1.18} 
  \widehat{{\mathcal E}}(t) = \frac{1}{(\ln \det G)_t \sqrt{\det G}}
          {\mathcal E}(t).
\end{equation} 
Using (\ref{A.18}), one finds
\begin{equation} \label{1.19} 
  \frac{d\widehat{{\mathcal E}}}{dt} \: = \: - \: \frac{1}{\sqrt{\det G}}
  \int_{Y} L^{-1} \Tr \left(
  \left( G^{-1} G_t \right)^2 \right) \dvol.
\end{equation}

Since $\widehat{{\mathcal E}}$ is nonincreasing in time, it can be used to
understand the long time behavior of a Gowdy solution.  In order to
understand the short time behavior, we want a quantity that is
monotonically nondecreasing in time.  To find such a quantity,
note that the right-hand sides of (\ref{A.15}) and (\ref{1.19}) have a
roughly similar term.  This suggests using a different prefactor of
${\mathcal E}$, as compared to (\ref{1.18}).
For this reason, we put
\begin{equation} \label{A.19}
  \widetilde{{\mathcal E}}(t) = \frac{\sqrt{\det G}}{(\ln \det G)_t}
          {\mathcal E}(t).
\end{equation}
Using (\ref{A.18}), one finds
\begin{equation} \label{A.20}
  \frac{d\widetilde{{\mathcal E}}}{dt} \: = \: \sqrt{\det G}
  \int_{Y} L h^{-1} \Tr \left(
  \left( G^{-1} G_y \right)^2 \right) \dvol.
\end{equation}
Hence $\widetilde{{\mathcal E}}$
is monotonically nondecreasing in $t$.
Note that the right-hand side of (\ref{1.19}) involves a time derivative,
whereas the right-hand side of (\ref{A.20}) involves a spatial derivative.

\begin{remark}
  If $N = 2$ and the Gowdy spacetime is polarized then the expression
  $\widetilde{{\mathcal E}}$ from (\ref{A.19}) is essentially the same
  as the expression $\epsilon^{(1)}$ from
  \cite[(19)]{Isenberg-Moncrief (1990)}.
  \end{remark}

\begin{remark} \label{1.22}
  We correct a couple of equations in \cite{Lott (2018)}.
  Equation (A.18) should not have the $\frac12$ on the right-hand side.
  The right-hand side of (A.20) should be multiplied by two.
\end{remark}

As $\widetilde{{\mathcal E}}$
is nonnegative and nondecreasing in $t$, it follows from (\ref{A.20}) that
for any $t_0 > 0$,
\begin{equation} \label{1.23}
  \int_0^{t_0} \int_Y \sqrt{\det G}
 L h^{-1} \Tr \left(
  \left( G^{-1} G_y \right)^2 \right) \dvol \: dt \: < \: \infty.
\end{equation}
As we will use the fact that $G^{- \: \frac12} G_y G^{- \: \frac12}$ is a symmetric matrix, we rewrite (\ref{1.23}) as
\begin{equation} \label{1.24}
  \int_0^{t_0} \int_Y \sqrt{\det G}
 L h^{-1} \Tr \left(
  \left( G^{- \: \frac12} G_y G^{- \: \frac12} \right)^2 \right) \dvol \: dt \: < \: \infty.
\end{equation}

We can rewrite (\ref{A.17}) as
\begin{equation} \label{1.25}
  \partial_t \left( \sqrt{\det G} L^{-1} h^{\frac12} G^{-1} G_t \right) =
  \partial_y \left(  \sqrt{\det G} L h^{- \: \frac12} G^{-1} G_y \right).
\end{equation}
Let $\sigma$ be a self-adjoint endomorphism of the vector bundle $e$,
with compact support in $(0, t_0) \times Y$. Locally,
\begin{equation} \label{1.26}
  \sigma^T = G \sigma G^{-1}.
\end{equation}
We note that $G^{-1} G_t$ and  $G^{-1} G_y$ are self-adjoint in this sense.
We write $\sigma_t$ for $\partial_t \sigma$, and similarly for $\sigma_y$.
Multiplying (\ref{1.25}) by $\sigma$, taking the trace and integrating by parts gives
\begin{align} \label{1.27}
& \int_0^{t_0} \int_Y \Tr \left(  \sigma_t \sqrt{\det G} L^{-1} h^{\frac12} (G^{-1} G_t) \right) dy \: dt = \\
& \int_0^{t_0} \int_Y \Tr \left(  \sigma_y  \sqrt{\det G} L h^{- \: \frac12} (G^{-1} G_y) \right)  dy \: dt. \notag
\end{align}
In terms of the density $\mu$ from (\ref{1.14}), this says
\begin{align} \label{1.28}
& \int_0^{t_0} \int_Y \frac{1}{(\ln \det G)_t} \Tr \left(  \sigma_t  (G^{-1} G_t) \right) d\mu \: dt = \\
& \int_0^{t_0} \int_Y \sqrt{\det G} L h^{- 1} \Tr \left(  \sigma_y (G^{-1} G_y) \right)  \dvol \: dt. \notag
\end{align}
Now
\begin{align} \label{1.29}
& \Tr \left(  \sigma_y (G^{-1} G_y) \right) = \\
  & \Tr \left(  (G^{\frac12} \sigma_y G^{- \: \frac12}) (G^{- \: \frac12}
  G_y G^{- \: \frac12}) \right) = \notag \\
  & \Tr \left( \frac12 \left(   (G^{\frac12} \sigma_y G^{- \: \frac12}) +
  (G^{\frac12} \sigma_y G^{- \: \frac12})^T \right) (G^{- \: \frac12}
  G_y G^{- \: \frac12}) \right) \notag,
\end{align}
where we use the symmetry of $G^{- \: \frac12}
  G_y G^{- \: \frac12}$ in the last step.
  Differentiating (\ref{1.26}) with respect to $y$, one can check that
  \begin{align} \label{1.30}
    \frac12 \left(   (G^{\frac12} \sigma_y G^{- \: \frac12}) +
    (G^{\frac12} \sigma_y G^{- \: \frac12})^T \right) = \: &
    \frac12 \left(   (G^{\frac12} \sigma_y G^{- \: \frac12}) +
    (G^{-\frac12} \sigma_y^T G^{\frac12}) \right) \\ = \: &
    G^{\frac12} (D_y \sigma) G^{- \: \frac12}, \notag
    \end{align}
  where
  \begin{equation} \label{1.31}
    D_y \sigma = \sigma_y + \frac12 \left[ G^{-1} G_y, \sigma \right].
  \end{equation}
  From (\ref{1.30}),  $G^{\frac12} (D_y \sigma) G^{- \: \frac12}$ is symmetric, which implies that
  $D_y \sigma$ is self-adjoint in the sense of (\ref{1.26}).
  Combining (\ref{1.28}), (\ref{1.29}) and (\ref{1.30}) gives
  \begin{align} \label{1.32}
& \int_0^{t_0} \int_Y \frac{1}{(\ln \det G)_t} \Tr \left(  \sigma_t  (G^{-1} G_t) \right) d\mu \: dt = \\
    & \int_0^{t_0} \int_Y \sqrt{\det G} L h^{- 1} \Tr \left(  (G^{\frac12} (D_y \sigma) G^{- \: \frac12})
    (G^{- \:  \frac12} G_y G^{\frac12}) \right)  \dvol \: dt. \notag
  \end{align}
  
By the Cauchy-Schwarz inequality,
and letting $c$ denote the square root of the left-hand side of (\ref{1.24}), we have
\begin{align} \label{1.33}
  &  \left|
  \int_0^{t_0} \int_Y \sqrt{\det G} L h^{- 1} \Tr \left(  (G^{\frac12} (D_y \sigma) G^{- \: \frac12})
    (G^{- \:  \frac12} G_y G^{\frac12}) \right)  \dvol \: dt
  \right| \le \\
  & c \: \sqrt{
\int_0^{t_0} \int_Y \sqrt{\det G} L h^{- 1} \Tr \left(  (G^{\frac12} (D_y \sigma) G^{- \: \frac12})^2  \right)  \dvol \: dt
  } = \notag \\
  & c \: \sqrt{
\int_0^{t_0} \int_Y \sqrt{\det G} L h^{- 1} \Tr \left(  (D_y \sigma)^2  \right)  \dvol \: dt
  }. \notag 
\end{align}

We can write
\begin{align} \label{1.34}
&  \int_0^{t_0} \int_Y \sqrt{\det G} L h^{- 1} \Tr \left(  (D_y \sigma)^2  \right)  \dvol \: dt \: = \\
&    \int_0^{t_0} \int_Y (\det G) (\ln \det G)_t \frac{\Tr \left(  (D_y \sigma)^2  \right) \: dy^2}{\mu}  \: dt. \notag
\end{align}
To put this in a more invariant way, write
\begin{equation} \label{1.35}
\Tr \left( D_Y \sigma \otimes D_Y \sigma  \right) =
\Tr \left(  (D_y \sigma)^2  \right) \: dy^2,
\end{equation}
a $2$-density on $Y$.  Then
$\frac{\Tr \left( D_Y \sigma \otimes D_Y \sigma   \right)}{\mu}$ is a density on $Y$ and (\ref{1.32})-(\ref{1.35}) give
\begin{align} \label{1.36}
  &  \left|
\int_0^{t_0} \int_Y \frac{1}{(\ln \det G)_t} \Tr \left(  \sigma_t (G^{-1} G_t) \right) d\mu \: dt
  \right| \le \\
  & c \: \sqrt{
    \int_0^{t_0} \int_Y (\det G)
    (\ln \det G)_t \frac{\Tr \left( D_Y \sigma \otimes D_Y \sigma   \right)}{\mu}  \: dt
  }, \notag
  \end{align}

We choose the time parameter $t$ so that $\det G = t^N$. Then
(\ref{1.36}) becomes
\begin{equation} \label{1.37}
 \left|
\int_0^{t_0} \int_Y t \Tr \left(  \sigma_t  (G^{-1} G_t) \right) d\mu \: dt
  \right| \le \\
N^{\frac32} c \: \sqrt{
  \int_0^{t_0} t^N
  \int_Y \frac{\Tr \left( D_Y \sigma \otimes D_Y \sigma   \right)}{\mu}  \: \frac{dt}{t}
  }.
\end{equation}
Next, we define $\tau$ by $t = e^{- \: \tau}$, so that
approaching the singularity corresponds to $\tau \rightarrow \infty$.
Define $\tau_0$ by $t_0 = e^{- \: \tau_0}$. Then (\ref{1.37}) becomes
\begin{equation} \label{1.38}
 \left|
\int_{\tau_0}^\infty \int_Y \Tr \left(  \sigma_\tau  (G^{-1} G_\tau) \right) d\mu \: d\tau
  \right| \le \\
N^{\frac32} c \: \sqrt{
  \int_{\tau_0}^\infty
  e^{- N \tau}
  \int_Y \frac{\Tr \left( D_Y \sigma \otimes D_Y \sigma   \right)}{\mu}  \: d\tau
  },
\end{equation}
or
\begin{equation} \label{1.39}
 \left|
\int_{\tau_0}^\infty \int_Y \Tr \left(  \sigma  (G^{-1} G_\tau)_\tau \right) d\mu \: d\tau
  \right| \le \\
N^{\frac32} c \: \sqrt{
  \int_{\tau_0}^\infty
  e^{- N \tau}
  \int_Y \frac{\Tr \left( D_Y \sigma \otimes D_Y \sigma   \right)}{\mu}  \: d\tau
  },
\end{equation}

From (\ref{1.32}), $(G^{-1} G_\tau)_\tau$ is orthogonal to $\Ker(D_Y)$ at all times $\tau$.
Define a time-$\tau$ norm $\| \cdot \|_{H^{-1}_{Y,\tau}}$ on 
square-integrable self-adjoint sections of $e$ by
\begin{equation} \label{1.40}
   \| \eta \|_{H^{-1}_{Y,\tau}} =  \sup_{\widehat{\eta} \neq 0}
\left( \int_Y \Tr \left(  \eta \widehat{\eta} \right) d\mu \right)
  \Big/
\sqrt{
      \int_Y \left( \Tr(\widehat{\eta}^2) d\mu + \frac{\Tr \left( D_Y \widehat{\eta} \otimes D_Y \widehat{\eta}   \right)}{\mu}
      \right)
}.
\end{equation}
where $\widehat{\eta}$ ranges over smooth time-$\tau$ self-adjoint sections of $e$.
Note that $\| \cdot \|_{H^{-1}_{Y,\tau}}$ depends on $G(\tau)$ through the notion of self-adjointness,
but only depends on $L(\tau)$ and $h(\tau)$ through the $\tau$-independent density $\mu$.
Let $H^{-1}_{Y,\tau}$ be the metric completion with respect to $\| \cdot \|_{H^{-1}_{Y,\tau}}$.
Let ${\mathcal H}$ be the weighted Hilbert space of measurable maps $f$ with
$f(\tau) \in H^{-1}_{Y,\tau}$ such that
\begin{equation} \label{1.41}
\int_0^\infty e^{N\tau}  \| f(\tau) \|_{H^{-1}_{Y,\tau}}^2 \: d\tau < \infty.
\end{equation}

\begin{proposition} \label{1.42}
\begin{equation} \label{1.43}
  \int_{\tau_0}^\infty e^{N\tau} \|
  (G^{-1} G_\tau)_\tau \|^2_{H^{-1}_{Y,\tau}} \: d\tau < \infty.
\end{equation}
\end{proposition}
\begin{proof}
  Equation (\ref{1.39}) implies that $(G^{-1} G_\tau)_\tau \in
  {\mathcal H}$.  The theorem follows.
  \end{proof}

\subsection{NonGowdy spacetime} \label{subsect2.4}

We now assume that $F \neq 0$.
If $\overline{R}_{I \alpha}^{\overline{g}} = 0 $ then
from the second equation in (\ref{A.3}), one finds that the
$\R^N$-valued vector
\begin{equation} \label{A.23}
  C_I = L^{-1} h^{- \: \frac12} \: \sqrt{\det G} \: G_{IK} F^K_{ty}
\end{equation}
is locally constant on the two dimensional spacetime.  More
precisely, it is a locally constant section of the flat vector bundle
$e^*$ (using our assumption that $e$ is unimodular).

We now restrict to the case when $N = 2$ and the flat $\R^2$-bundle $e$
has holonomy ${\mathcal T}$, around the circle $Y$, lying in $\SL(2, \R)$.
When ${\mathcal T} = \Id$, the components of $C$ are
called the ``twist quantities'' in
\cite{Berger-Chrusciel-Isenberg-Moncrief (1997)} and subsequent papers
such as \cite{LeFloch-Smulevici (2015)}.
We mostly follow the notation of \cite[p. 1256-1283]{LeFloch-Smulevici (2015)},
with coordinates $(R, \theta)$ for the two dimensional base. We use
linear coordinates
$x^1, x^2$ for the $\R^2$-fiber.  In that paper, $R = \sqrt{\det G}$ and $\theta$
is the local
coordinate for the spacelike hypersurface $Y$.  The coordinates
$x^1$ and $x^2$ are chosen so that $C_1 = 0$ and $C_2 = K$, where $K$ is
a positive constant. The Lorentzian metric on
$(0, \infty) \times Y$ can be written as
\begin{equation} \label{A.24}
  g = e^{2(\eta - U)} (-dR^2 + a^{-2} d\theta^2) +
  e^{2U} (dx^1 + Adx^2 + (G+AH) d\theta)^2 +
  e^{-2U}R^2(dx^2+Hd\theta)^2.
\end{equation}
The variables $\eta$, $U$, $a$, $A$, $G$ and $H$ are functions of $R$ and
$\theta$.
To relate to (\ref{A.1}), the inner product $G_{IJ}$ is
\begin{equation} \label{1.46}
\begin{pmatrix}
  e^{2U} & e^{2U} A \\
  e^{2U} A & e^{2U} A^2 + e^{-2U} R^2
\end{pmatrix}
\end{equation}
and the connection $A^I$ is
\begin{equation} \label{1.47}
\begin{pmatrix}
  G d\theta \\
  H d\theta
  \end{pmatrix}.
\end{equation}

The analog of the density $\mu$ from (\ref{1.14}) is
$2 a^{-1} \: d\theta$; it is no longer independent of the
time parameter $R$.

Put
\begin{equation} \label{A.25}
  {\mathcal D} =
  a^{-1} U_R^2 + a U_\theta^2 + \frac14 R^{-2} e^{4U}
  (a^{-1} A_R^2 + a A_\theta^2)
\end{equation}
and
\begin{equation} \label{A.26}
  \widehat{\mathcal E}_K(R) = \int_{Y} \left( {\mathcal D}
  + \frac14 K^2 R^{-4} e^{2\eta} a^{-1} \right) d\theta.
\end{equation}
Then from \cite[p. 1283]{LeFloch-Smulevici (2015)},
\begin{equation} \label{A.27}
  \frac{d\widehat{\mathcal E}_K}{dR} =
  - 2 R^{-1} \int_{Y} \left( a^{-1} U_R^2 + \frac14
  R^{-2} e^{4U} a A_\theta^2 \right) d\theta - K^2 R^{-5}
  \int_{Y} a^{-1} e^{2 \eta} \: d\theta.
\end{equation}
(There were some incorrect terms in \cite[(A.25) and (A.27)]{Lott (2018)}.)

Put
\begin{equation} \label{1.51}
  \widetilde{\mathcal E}_K = R^2 \widehat{\mathcal E}_K + \frac12 K \int_Y H \: d\theta.
\end{equation}
As seen from (\ref{1.47}), the term $\int_Y H \: d\theta$ is the holonomy of the twist component
of the connection $A^I$.

\begin{proposition} \label{1.52} We have
  \begin{equation} \label{1.53}
     \frac{d\widetilde{\mathcal E}_K}{dR} =
2R \int_Y \left( a U_\theta^2 + \frac14 R^{-2} e^{4U} a^{-1} A_R^2 \right) \: d\theta.
\end{equation}
    \end{proposition}
\begin{proof}
  Using (\ref{1.51}) and (\ref{1.53}),
  \begin{align} \label{1.54}
    \frac{d}{dR} \left( R^2 \widehat{\mathcal E}_K \right) = & 
    R^2 \left( \frac{d\widehat{\mathcal E}_K}{dR} + 2 R^{-1} \widehat{\mathcal E}_K \right) \\
    = & 
    2R \int_Y \left( a U_\theta^2 + \frac14 R^{-2} e^{4U} a^{-1} A_R^2 \right) \: d\theta
    - \: \frac12 K^2 R^{-3}
    \int_{Y} a^{-1} e^{2 \eta} \: d\theta. \notag
  \end{align}
  From \cite[(4.28)]{LeFloch-Smulevici (2015)},
  \begin{equation} \label{1.55}
    \frac{\partial H}{\partial R} = K R^{-3} a^{-1} e^{2 \eta}.   
  \end{equation}
  The proposition follows.
  \end{proof}

Proposition \ref{1.52} is also valid if the holonomy
${\mathcal T}$ is such that
${\mathcal T}^{-T}$ is unipotent; c.f. \cite[Section A.3.2]{Lott (2018)}.

Suppose that $\int_Y H \: d\theta$ is uniformly bounded below as $R \rightarrow 0$.
Then $\widetilde{\mathcal E}_K$ is uniformly bounded below and 
Proposition \ref{1.52} implies that
\begin{equation} \label{1.56}
 \int_0^{R_0} \int_Y  \left( R a U_\theta^2 + \frac14 R^{-1} e^{4U} a^{-1} A_R^2 \right) \: d\theta \: dR \: < \: \infty.
  \end{equation}

From \cite[(4.22)]{LeFloch-Smulevici (2015)},
\begin{equation} \label{1.57}
(Ra^{-1}U_R)_R - (RaU_\theta)_\theta = \frac12 R^{-1} e^{4U} \left( a^{-1} A_R^2 - a A_\theta^2 \right).
\end{equation}
Let $\sigma$ be a smooth function with compact support in $(0,R_0) \times S^1$. Multiplying (\ref{1.57}) by $\sigma$ and
integrating gives
\begin{align} \label{1.58}
  &  \int_0^{R_0} \int_Y \sigma \left( (Ra^{-1}U_R)_R
  \: - \: \frac12 R^{-1} e^{4U} a^{-1} A_R^2 +
  \frac12 R^{-1} e^{4U} a A_\theta^2  \right) \: d\theta \: dR = \\
  &  \int_0^{R_0} \int_Y \sigma (RaU_\theta)_\theta \: d\theta \: dR \: = \: 
- \: \int_0^{R_0} \int_Y \sigma_\theta RaU_\theta \: d\theta \: dR.  \notag
\end{align}
Letting $c$ denote the square root of the left-hand side of (\ref{1.56}), the Cauchy-Schwarz inequality gives
\begin{align} \label{1.59}
  & \left| \int_0^{R_0} \int_Y \sigma \left( a(Ra^{-1}U_R)_R
   \: - \: \frac12 R^{-1} e^{4U} A_R^2 +
  \frac12 R^{-1} e^{4U} a^2 A_\theta^2\right) \: a^{-1} d\theta \: dR \right| \le \\
&  c \: \sqrt{
    \int_0^{R_0} \int_Y  R a^2 \sigma_\theta^2 \: a^{-1} d\theta \: dR}. \notag
\end{align}

Changing variable by $R = e^{- \tau}$ gives
\begin{align} \label{1.60}
  & \left| \int_{\tau_0}^\infty \int_Y \sigma \left( a(a^{-1}U_\tau)_\tau
   \: - \: \frac12 e^{2 \tau} e^{4U} A_\tau^2 +
   \frac12 e^{4U} a^2 A_\theta^2\right) \: a^{-1} d\theta \: d\tau
   \right| \le \\
&  c \: \sqrt{
     \int_{\tau_0}^\infty \int_Y  e^{- 2 \tau} a^2 \sigma_\theta^2 \: a^{-1}
     d\theta \: d\tau}. \notag
\end{align}

Define a time-$\tau$ norm $\| \cdot \|_{H^{-1}_{Y,\tau}}$ on 
$L^2(Y; a^{-1} d\theta)$ by
\begin{equation} \label{1.61}
   \| \sigma \|_{H^{-1}_{Y,\tau}} = \sup_{\widehat{\sigma} \neq 0}
   \left( \int_Y \sigma \widehat{\sigma} a^{-1} d\theta \right)
  \Big/
\sqrt{
  \int_Y \left( \widehat{\sigma}^2  + a^2 \widehat{\sigma}_\theta^2
  \right) a^{-1} d\theta
},
\end{equation}
where $\widehat{\sigma}$ ranges over smooth functions on $Y$.
Let $H^{-1}_{Y,\tau}$ be the metric completion with respect to
$\| \cdot \|_{H^{-1}_{Y,\tau}}$.
Let ${\mathcal H}$ be the weighted Hilbert space of measurable maps $f$ with
$f(\tau) \in H^{-1}_{Y,\tau}$ such that
\begin{equation} \label{1.62}
\int_{\tau_0}^\infty e^{2\tau}  \| f(\tau) \|_{H^{-1}_{Y,\tau}}^2 \: d\tau < \infty.
\end{equation}
Then (\ref{1.60}) implies that
\begin{equation} \label{1.63}
a(a^{-1}U_\tau)_\tau
   \: - \: \frac12 e^{2 \tau} e^{4U} A_\tau^2 +
   \frac12 e^{4U} a^2 A_\theta^2 \in {\mathcal H}.
   \end{equation}

The AVTD asymptotics for $U$ are that $a(a^{-1}U_\tau)_\tau
\: - \: \frac12 e^{2 \tau} e^{4U} A_\tau^2$ goes to zero as
$\tau \rightarrow \infty$
\cite{Ames-Beyer-Isenberg-LeFloch (2013),Clausen-Isenberg (2007)}.
The next theorem gives a sufficient condition for
AVTD asymptotics to hold for $U$, in an integral sense.

\begin{proposition} \label{1.64}
If $\int_{S^1} H \: d\theta$ is bounded below as $\tau \rightarrow \infty$ and
\begin{equation} \label{1.65}
  e^{4U} a^2 A_\theta^2 \in {\mathcal H}
\end{equation}
then 
\begin{equation} \label{1.66}
a(a^{-1}U_\tau)_\tau
\: - \: \frac12 e^{2 \tau} e^{4U} A_\tau^2 \in {\mathcal H}.
\end{equation}
\end{proposition}
\begin{proof}
This follows from  (\ref{1.63}).
\end{proof}

\begin{remark} \label{1.67}
The formal large-$\tau$ asymptotics from
\cite[(3.8)-(3.13)]{Ames-Beyer-Isenberg-LeFloch (2013)} say
\begin{align} \label{1.68}
  U(\tau,\theta) \sim & - \: \frac{1-k(\theta)}{2} \tau +
    U_{\star \star}(\theta) + \ldots, \\
  A(\tau,\theta) \sim &  A_{\star}(\theta)
  + A_{\star \star}(\theta) e^{-2k(\theta)\tau}
  + \ldots, \notag \\
  a(\tau,\theta) \sim & a_{\star}(\theta) + \ldots \notag \\
    H(\tau,\theta) \sim & H_{\star}(\theta) + \ldots \notag
    \end{align}
where $k(\theta)$ determines the Kasner parameters.
Without further assumptions, one sees that
(\ref{1.65}) should not always hold. On the other hand,
if we assume that $A_{\star}$ is constant in
$\theta$, the half-polarized condition, then
$e^{4U} a^2 A_\theta^2 \sim e^{- 2 \tau}$ and
(\ref{1.65}) holds. This is consistent with the finding in
\cite{Ames-Beyer-Isenberg-LeFloch (2013)} that the half-polarized
condition is needed for the Fuchsian method to work.

Without the half-polarized condition, it appears from (\ref{1.63})
that the right generalization of the AVTD asymptotics for $U$ would
be that $a(a^{-1}U_\tau)_\tau
   \: - \: \frac12 e^{2 \tau} e^{4U} A_\tau^2 +
   \frac12 e^{4U} a^2 A_\theta^2$ goes to zero as $\tau \rightarrow
   \infty$. For general $T^2$-symmetric vacuum spacetimes,
   numerics indicate a Mixmaster-type behavior
   \cite{Berger-Isenberg-Weaver (2001)}.  Of the two
   conditions in Proposition \ref{1.64}, we do not know which ones are violated
   in such a case.
\end{remark}

\section{CMC Einstein flows} \label{sect3}

In this section we consider expanding vacuum spacetimes with a 
CMC foliation. In Subsection \ref{subsect3.1}
we discuss the quantity $t^{-1} \vol(X, h(t))$, and prove
Theorems \ref{0.9}, \ref{0.12} and \ref{0.13}.
In Subsection \ref{subsect3.2} we define
noncollapsing type-I Einstein flows and their
rescalings.  Subsection \ref{subsect3.3} has the proof of
Theorem \ref{0.15}. In Subsection \ref{subsect3.4} we prove Theorems
\ref{0.16} and \ref{0.17}. Subsection \ref{subsect3.5}
has an improvement of Theorem \ref{0.13} in the case
of a noncollapsed type-I Einstein flow.

\subsection{Monotonic quantities} \label{subsect3.1}

\begin{definition} \label{old1.1}
  Let $I$ be an interval in $\R$.
  An Einstein flow
  ${\mathcal E}$ on an $n$-dimensional manifold $X$ is given by a
  family of nonnegative functions $\{L(t)\}_{t \in I}$ on $X$,
  a family of Riemannian metrics $\{h(t)\}_{t \in I}$ on $X$, and a family
  of symmetric covariant $2$-tensor fields $\{K(t)\}_{t \in I}$ on $X$,
  so that if
  $H = h^{ij} K_{ij}$ and $K^0 = K - \frac{H}{n} h$ then
  the constraint equations
  \begin{equation} \label{old1.2}
    R - |K^0|^2 + \left( 1 - \frac{1}{n} \right) H^2 =  0
  \end{equation}
  and
  \begin{equation} \label{old1.3}
    \nabla_i K^i_{\: \: j} - \nabla_j H =  0,
  \end{equation}
  are satisfied, along with the evolution equations
  \begin{equation} \label{old1.4}
    \frac{\partial h_{ij}}{\partial t} = - 2 L K_{ij}
  \end{equation}
  and
  \begin{equation} \label{old1.5}
    \frac{\partial K_{ij}}{\partial t} =  L H K_{ij} - 2 L
    h^{kl} K_{ik} K_{lj} - L_{;ij} + L R_{ij}.
  \end{equation}
\end{definition}

For now, we will assume that $X$ is compact and connected, and
that all of the data
is smooth.
At the moment, $L$ is unconstrained; it will be determined by
the elliptic equation (\ref{old1.13}) below.
We will generally want $L(t)$ to be positive.

An Einstein flow gives rise to a Ricci-flat Lorentzian metric
\begin{equation} \label{old1.6}
  g = - L^2 dt^2 + h(t)
\end{equation}
on $I \times X$, for which the second fundamental form of the
time-$t$ slice is $K(t)$. Conversely, given a
Lorentzian metric $g$ on a manifold with a proper time function $t$,
we can write it in the
form (\ref{old1.6}) by using the flow of $\frac{\nabla t}{|\nabla t|^2}$ to
identify nearby leaves.
Letting $K(t)$ be the second fundamental form of
the time-$t$ slice, the metric $g$ is Ricci-flat if and only if
$(L,h,K)$ is an Einstein flow.

\begin{definition} \label{2.7} \cite{Eardley-Smarr (1979), Marsden-Tipler (1980)}
There is a {\em crushing singularity} as $t \rightarrow 0$ if
there is a
sequence of compact Cauchy surfaces going out the end at
$\{0\} \times X$ whose mean curvatures approach $- \infty$ uniformly.
\end{definition}

From \cite{Gerhardt (1983)}, if there is a crushing singularity then
there is a foliation near the end by
constant mean curvature (CMC) compact spatial hypersurfaces, whose mean
curvatures approach $- \infty$.

\begin{definition} \label{old1.7}
  A CMC Einstein flow is an Einstein flow
  for which $H$ only depends on $t$.
\end{definition}

We will assume that there is a crushing singularity as $t \rightarrow 0$;
in particular, $X$ is compact.
So we can assume we have a CMC Einstein flow with $I = (0, t_0]$
    (or $I = (0, t_0)$), and that $H$ is
    monotonically increasing in $t$
    and takes all values in
    $(- \infty, H_0)$ for some $H_0 < 0$.

    \begin{example} \label{2.9}
      We give some relevant examples of crushing singularities.
            \begin{enumerate}
      \item
        A Lorentzian cone over a Riemannian Einstein $n$-manifold
        $(X,h_{Ein})$ with
        Einstein constant $- (n-1)$. The metric is
\begin{equation} \label{2.10}
  g = - dt^2 + t^2 h_{Ein}.
\end{equation}
\item The product of the previous example, in dimension $n - n^\prime$, with
  a flat torus $(T^{n^\prime}, h_{flat})$. The metric is
  \begin{equation} \label{2.11}
  g = - dt^2 + t^2 h_{Ein} + h_{flat}.
  \end{equation}
\item A Kantowski-Sachs solution with $X$ diffeomorphic to
  $S^2 \times S^1$. The metric is a $\Z$-quotient of the interior of the
  event horizon in a Schwarzschild solution, after switching the
  usual $t$ and $r$ variables:
  \begin{equation} \label{2.12}
    g = - \:
    \frac{1}{\frac{2m}{t}-1} dt^2 + \left( \frac{2m}{t} - 1\right) dr^2
    + t^2 g_{S^2}.
  \end{equation}
  Here $t \in (0, 2m)$ and the $\Z$-quotienting is in the $r$-variable.
\item A Kasner solution on a flat $n$-manifold.  After possibly passing
  to a finite cover of $X$, the metric is
  \begin{equation} \label{2.13}
    g = - \: \frac{1}{n^2} dt^2 + (d \vec{x})^T t^{2M} d\vec{x}.
  \end{equation}
  Here $M$ is a symmetric $(n \times n)$-matrix with $\Tr(M) = \Tr(M^2) = 1$.
  We have written the metric so that $t \: = \: - \: \frac{n}{H}$. Then
  \begin{equation} \label{2.14}
    L = \frac{1}{n}, \: \: \: \: \: \:
    R = 0, \: \: \: \: \: \:
    |K|^2 = H^2 = \frac{n^2}{t^2}.
    \end{equation}
            \end{enumerate}
    \end{example}
    \noindent
    {\it End of example.} \\
    
Returning to general expanding CMC Einstein flows,
equation (\ref{old1.5}) gives
\begin{align} \label{old1.13}
  \frac{\partial H}{\partial t} = & - \triangle_h L + LH^2
  + LR \\
  = & - \triangle_h L + L |K^0|^2 + \frac{1}{n} LH^2. \notag
\end{align}
The maximum principle gives
\begin{equation} \label{old1.14}
  L(t)
  \le \frac{n}{H^2} \frac{\partial H}{\partial t}.
\end{equation}

\begin{proposition} \label{old1.15} \cite{Fischer-Moncrief (2002)}
  Let ${\mathcal E}$ be an expanding CMC Einstein flow.
  The quantity $(-H)^n \vol(X,h(t))$ is monotonically nonincreasing
  in $t$. It is constant in $t$ if and only if, taking
  $t = - \: \frac{n}{H}$, the Einstein flow ${\mathcal E}$ is
  a Lorentzian cone over a Riemannian Einstein manifold with
  Einstein constant $-(n-1)$.
  \end{proposition}

One proof of \cite{Fischer-Moncrief (2002)}
uses the pointwise identity
\begin{equation} \label{old1.16}
  \frac{\partial}{\partial t} \left( (-H)^n \dvol_h \right) =
  (-H)^{n+1} \left( L -  \frac{n}{H^2} \frac{\partial H}{\partial t}
  \right) \: \dvol_h.
\end{equation}
From (\ref{old1.14}),
it follows that $(-H)^n \dvol_{h(t)}$ is pointwise
monotonically nonincreasing
in $t$, and hence $(-H)^n \vol(X,h(t))$ is monotonically nonincreasing in $t$.
In fact,
\begin{equation} \label{added}
  \frac{d}{dt}  \left( (-H)^n \vol(X,h(t)) \right) =
  - n (-H)^{n-1} \int_X |K^0|^2 L \dvol_h.
  \end{equation}
If $n > 1$, we can use (\ref{old1.2}) to write the monotonic quantity
itself as
\begin{equation} \label{added2}
  (-H)^n \vol(X,h(t)) = \frac{n}{n-1} (-H)^{n-2}
  \int_X \left( - R^h + |K^0|^2 \right) \dvol_h.
  \end{equation}

The monotonic quantity $(-H)^n \vol(X,h(t))$ gives information about
the large time behavior of the expanding solution
\cite{Fischer-Moncrief (2002),Lott (2018)}.  To get information about
the small time behavior, we want a quantity that is
instead monotonically
nondecreasing in $t$. As discussed in Subsection \ref{subsect2.3}, we can
try to play the right-hand sides of (\ref{added}) and (\ref{added2})
against each other. The right quantity turns out to be
$(-H) \vol(X, h(t))$.

\begin{proposition} \label{2.19} We have
  \begin{equation} \label{2.20}
  \frac{\partial}{\partial t} \left( (-H) \dvol_h \right) =
  H^{2} \left( L -  \frac{1}{H^2} \frac{\partial H}{\partial t}
  \right) \: \dvol_h.
\end{equation}
\end{proposition}
\begin{proof}
  This follows from (\ref{old1.4}).
  \end{proof}

\begin{corollary} \label{2.21}
  \begin{equation} \label{2.22}
  \frac{d}{dt} \left( (-H) \vol(X, h(t)) \right) =
  - \int_X L R \: \dvol_h.
\end{equation}
  \end{corollary}
\begin{proof}
  From (\ref{old1.13}) and Proposition \ref{2.19}, we have
  \begin{equation} \label{2.23}
    \frac{d}{dt} \left( (-H) \vol(X, h(t)) \right) =
    \int_X (\triangle_h L - L R) \: \dvol_h =
    - \int_X L R \: \dvol_h.
  \end{equation}
This proves the claim.
\end{proof}

Note that as in Subsection \ref{subsect2.3},
the time derivative on the right-hand side of (\ref{added}) turns into the
spatial derivatives on the right-hand side of (\ref{2.22}).

\begin{corollary} \label{2.24}
  If $t_1 < t_2$ then
    \begin{equation} \label{2.25}
     (-H(t_2)) \vol(X, h(t_2)) - (-H(t_1)) \vol(X, h(t_1)) =
    - \int_{t_1}^{t_2} \int_X L R \: \dvol_{h(t)} \: dt.
  \end{equation}
\end{corollary}

\begin{corollary} \label{2.26}
  \begin{equation} \label{2.27}
    - \int_{0}^{t_0} \int_X L R \: \dvol_{h(t)} \: dt < \infty.
    \end{equation}
\end{corollary}  
\begin{proof}
  This follows from (\ref{2.25}) by taking $t_2 = t_0$ and $t_1 \rightarrow 0$,
  along with the fact that
  $(-H(t_1)) \vol(X, h(t_1)) \ge 0$.
  \end{proof}

\begin{corollary} \label{2.28}
  If $R \le 0$ then $(-H(t))
  \vol(X, h(t))$ is monotonically nondecreasing in $t$.
  \end{corollary}

\begin{example} \label{2.29}
  If $\dim(X) = 3$ and $X$ is aspherical then
  a locally homogenous Einstein flow on $X$ has $R \le 0$, since $X$ admits
  no metric of positive scalar curvature, so Corollary \ref{2.28} applies.
  \end{example}

\begin{proposition} \label{2.30}
If $R \le 0$ then $L \ge \frac{1}{H^2} \frac{\partial H}{\partial t}$.
\end{proposition}
\begin{proof}
  This follows from (\ref{old1.13}) and the weak maximum principle.
\end{proof}

We now improve Corollary \ref{2.28} to a pointwise statement.

\begin{corollary} \label{2.31}
  If $R \le 0$ then
  $(-H) \dvol_{h(t)}$ is monotonically nondecreasing in $t$.
  \end{corollary}
\begin{proof}
  This follows from Propositions \ref{2.19} and \ref{2.30}.
  \end{proof}

\begin{proposition} \label{2.32}
  If $R \le 0$ and
  $L(x,t) = \frac{1}{H^2} \frac{\partial H}{\partial t}$
  for some $x$ and $t$, then $L = \frac{1}{H^2} \frac{\partial H}{\partial t}$,
  $R = 0$ and
$|K|^2 = H^2$.
\end{proposition}
\begin{proof}
  Equation (\ref{old1.13}) and the strong maximum principle imply that
  $L = \frac{1}{H^2} \frac{\partial H}{\partial t}$
  and $R = 0$. Equation (\ref{old1.2}) gives
  $|K|^2 = H^2$
\end{proof}

\begin{remark} \label{2.33}
  The conclusion of Proposition \ref{2.32} does not use the compactness of $X$, or
  even the completeness of $(X, h(t))$.
  \end{remark}

\begin{proposition} \label{2.34}
  If $R \le 0$ and $(-H(t_1)) \vol(X, h(t_1)) = (-H(t_2)) \vol(X, h(t_2))$,
  for some $t_1 < t_2$, then $L = \frac{1}{H^2} \frac{\partial H}{\partial t}$,
  $R = 0$ and
  $|K|^2 = H^2$.
\end{proposition}
\begin{proof}
  From Propositions \ref{2.19} and \ref{2.30}, we know that
  $L = \frac{1}{H^2} \frac{\partial H}{\partial t}$. The other claims follow
  from Proposition \ref{2.32}.
  \end{proof}

\begin{proposition} \label{2.35}
  Under the hypotheses of Proposition \ref{2.34}, suppose that
  $X$ is an orientable $3$-manifold and that
  there is an aspherical component in the prime decomposition of $X$.
  Then up to time reparametrization,
  the Einstein flow is a Kasner solution. 
  \end{proposition}
\begin{proof}
  We know that $R = 0$.
  Running the Ricci flow with initial condition $h(t)$, either the scalar
  curvature becomes immediately positive or $h(t)$ is Ricci-flat.
  One knows that $X$ admits no metric with positive scalar curvature
  \cite{Schoen-Yau (1979)}. Hence $h(t)$ is Ricci-flat. Because $\dim(X) = 3$,
  the metric $h(t)$ is flat.

  In matrix notation, equations (\ref{old1.4}) and (\ref{old1.5}) become
  \begin{align} \label{2.36}
    \frac{dh}{dt} = & - 2 L K, \\
    \frac{dK}{dt} = & LHK-2LKh^{-1}K. \notag
    \end{align}
 Then $h^{-1} K$ satifies
 \begin{equation} \label{2.37}
\frac{d}{dt} (h^{-1} K) = LH h^{-1}K = \frac{1}{H} \frac{dH}{dt} h^{-1} K,
 \end{equation}
 with the general solution
 \begin{equation} \label{2.38}
   h^{-1} K = HM,
 \end{equation}
 where $M$ is a time-independent self-adjoint section of $\End(TX)$.
 Equation (\ref{old1.2}) gives $\Tr(M^2) = 1$.
Also, the fact that $H = \Tr(h^{-1} K)$ gives $\Tr(M) = 1$.
Since $X$ is flat, equation (\ref{old1.3}) means
that $M$ is
locally constant. 
Then
\begin{equation} \label{2.39}
  \frac{dh}{dt} = - 2LK = -2 \frac{1}{H} \frac{dH}{dt} h M,
\end{equation}
with the general solution
\begin{equation} \label{2.40}
  h = \widehat{h} (-H)^{-2M}
\end{equation}
for some time-independent metric $\widehat{h}$ on $X$. This is a Kasner
solution.
\end{proof}

For $s > 1$, the Lorentzian metric $s^{-2} g$ is isometric to
\begin{equation} \label{old1.35}
  g_s = - L^2(su) du^2 + s^{-2} h(su).
\end{equation}
Hence we put
\begin{align} \label{old1.36}
  & L_s(u) =  L(su),
  & h_s(u) = s^{-2} h(su), \: \: \:
  & K_{s,ij}(u) = s^{-1} K_{ij}(su),\\
  & H_s(u) = s H(su),
  & K^0_{s,ij}(u) = s^{-1} K^0_{ij}(su), \: \: \:
  & |K^0|^2_{s}(u) = s^2 K_{ij}(su), \notag \\
  & R_{s,ij}(u) = R_{ij}(su),
  & R_s(u) = s^2 R(su).
  & \notag
\end{align}
The variable $u$ will refer to the time parameter of a rescaled Einstein
flow, or a limit of such.

We now take $t = - \frac{n}{H}$, the Hubble time,
with $t$ ranging in an interval $(0, t_0]$.
Equation (\ref{2.27}) becomes
  \begin{equation} \label{2.43}
    \int_0^{t_0} \int_X (- t^2 R) \: L \: \frac{\dvol_{h(t)}}{t}
      \: \frac{dt}{t} < \infty. 
  \end{equation}
If $R \le 0$ then equation
(\ref{old1.14}) and Proposition \ref{2.30} say $\frac{1}{n} \le L \le 1$.
  
\begin{lemma} \label{2.44}
  If $R \le 0$ then
  $\vol(X, h(t)) = O (t)$ as $t \rightarrow 0$.
  \end{lemma}
  \begin{proof}
    This follows from Corollary \ref{2.24}.
  \end{proof}
  
\begin{definition} \label{2.45}
If $R \le 0$, put $\dvol_0 = \lim_{t \rightarrow 0} ((-H) \dvol_{h(t)})$.
\end{definition}

From Corollary \ref{2.31}, the definition of $\dvol_0$ makes sense.
It is
a nonnegative absolutely continuous measure on $X$. It could be zero.
  
The next proposition says that in an $L^1$-sense, rescaling limits
are similar to Kasner solutions; c.f. (\ref{2.14}).

\begin{proposition} \label{2.46}
  Suppose that $R \le 0$. Given $\Lambda > 1$, we have
\begin{equation} \label{2.47}
  \lim_{s \rightarrow 0} \left| L_s - \frac{1}{n} \right| =
  \lim_{s \rightarrow 0}
  \left| |K_s|^2 - \frac{n^2}{u^2} \right| = \lim_{s \rightarrow 0} |R_s| = 0
\end{equation}
in $L^1 \left( [\Lambda^{-1}, \Lambda] \times X, du \dvol_0 \right)$.
  \end{proposition}
\begin{proof}
  The proof is similar to that of \cite[Proposition 2.36]{Lott (2018)}.
  We omit the details.
\end{proof}

\subsection{Rescaling limits} \label{subsect3.2}

Let ${\mathcal E}$ be an Einstein flow. Let $g$ be the corresponding
Lorentzian metric.
Put $e_0 = T = \frac{1}{L} \frac{\partial}{\partial t}$,
a unit timelike vector
that is normal to the level sets of $t$.
Let $\{e_i\}_{i=1}^n$ be an orthonormal basis for $e_0^\perp$. Put
\begin{equation} \label{old1.47}
  |\Rm|_T = \sqrt{\sum_{\alpha, \beta, \gamma, \delta = 0}^n
    R_{\alpha \beta \gamma \delta}^2}.
\end{equation}

Let ${\mathcal E}^\infty = \left( L^\infty, h^{\infty}, K^{\infty}
\right) $ be a CMC Einstein flow on a pointed $n$-manifold
$\left( X^\infty, x^\infty \right)$,
with complete time slices, defined on a time interval
$I^\infty$. For the moment, $t$ need not be the Hubble time.
Take $p \in [1, \infty)$ and $\alpha \in (0,1)$.

\begin{definition} \label{old1.48}
  The flow ${\mathcal E}^\infty$ is $W^{2,p}$-regular if
  $X^\infty$ is a $W^{3,p}$-manifold,
  $L^\infty$ and $h^\infty$ are locally $W^{2,p}$-regular in space and time, and
  $K^\infty$ is locally $W^{1,p}$-regular in space and time.
\end{definition}
Note that the equations of Definition \ref{old1.1} make sense in this
generality.

Let ${\mathcal E}^{(k)} = \{h^{(k)}, K^{(k)}, L^{(k)} \}_{k=1}^\infty$ be
smooth CMC
Einstein flows on pointed $n$-manifolds
$\{ \left( X^{(k)}, x^{(k)} \right) \}_{k=1}^\infty$,
defined on time intervals $I^{(k)}$.

\begin{definition} \label{old1.49}
  We say that $\lim_{k \rightarrow \infty} {\mathcal E}^{(k)} =
  {\mathcal E}^\infty$ in the pointed weak $W^{2,p}$-topology
  if
  \begin{itemize}
  \item Any compact interval $S \subset I^\infty$ is contained in
    $I^{(k)}$ for large $k$, and
  \item For any compact interval $S \subset I^\infty$ and
    any compact $n$-dimensional manifold-with-boundary $W^\infty \subset X^\infty$
    containing $x^\infty$, for large $k$
    there are pointed time-independent $W^{3,p}$-regular diffeomorphisms
    $\phi_{S,W,k} : W^\infty \rightarrow W^{(k)}$ (with
    $W^{(k)} \subset X^{(k)}$) so that
    \begin{itemize}
    \item $\lim_{k \rightarrow \infty}
      \phi_{S,W,k}^* L^{(k)} = L^\infty$ weakly in $W^{2,p}$ on $S \times W^\infty$,
    \item $\lim_{k \rightarrow \infty}
      \phi_{S,W,k}^* h^{(k)} = h^\infty$ weakly in $W^{2,p}$ on $S \times W^\infty$
      and
    \item $\lim_{k \rightarrow \infty}
      \phi_{S,W,k}^* K^{(k)} = K^\infty$ weakly in $W^{1,p}$ on $S \times W^\infty$.
    \end{itemize}
  \end{itemize}
\end{definition}

We define pointed (norm) $C^{1,\alpha}$-convergence similarly.

\begin{definition} \label{old1.50}
  Let ${\mathcal S}$ be a collection of pointed CMC
  Einstein flows defined on a time interval $I^\infty$.
  We say that a sequence
  $\{ {\mathcal E}^{(k)} \}_{k=1}^\infty$ of pointed CMC
  Einstein flows approaches ${\mathcal S}$
  as $k \rightarrow \infty$, in the pointed weak $W^{2,p}$-topology, if
  for any subsequence of $\{ {\mathcal E}^{(k)} \}_{k=1}^\infty$, there
  is a further subsequence
  that converges to an element of ${\mathcal S}$
  in the pointed weak $W^{2,p}$-topology.
\end{definition}

\begin{definition} \label{old1.51}
  Let ${\mathcal S}$ be a collection of pointed CMC
  Einstein flows defined on a time
  interval $I^\infty$. We say that a $1$-parameter family
  $\{ {\mathcal E}^{(s)} \}_{s \in (0,s_0]}$ of pointed CMC Einstein
    flows
    approaches
    ${\mathcal S}$, in the pointed weak $W^{2,p}$-topology, if
    for any sequence $\{s_k\}_{k=1}^\infty$ in $(0, s_0]$ with
      $\lim_{k \rightarrow \infty} s_k = 0$, there is a subsequence
      of the flows $\{ {\mathcal E}^{(s_k)} \}_{k=1}^\infty$
      that converges to an element of ${\mathcal S}$
      in the pointed weak $W^{2,p}$-topology.
\end{definition}

We define ``approaches ${\mathcal S}$'' in the
pointed (norm) $C^{1,\alpha}$-topology similarly.
The motivation for these definitions comes from how one can
define convergence to a compact subset of a metric space, just
using the notion of sequential convergence. In our applications,
the relevant set ${\mathcal S}$ of Einstein flows can be taken to be
sequentially compact.

\begin{definition} \label{2.53}
We say that a pointed CMC
  Einstein flow
${\mathcal E}^1$ is
$\epsilon$-close to a  pointed CMC
  Einstein flow ${\mathcal E}^2$ in the pointed
$C^{1,\alpha}$-topology
if they are both defined on the time interval $(\epsilon, \epsilon^{-1})$ and, up to applying time-independent
pointed diffeomorphisms, 
the metrics are $\epsilon$-close in the $C^{1,\alpha}$-norm on $(\epsilon, \epsilon^{-1}) \times
B_{h_2(1)}(x^{(2)}, \epsilon^{-1})$.
\end{definition}

We don't make a similar definition of closeness for the
pointed weak $W^{2,p}$-topology because the weak topology is not
metrizable.

We now take $t = - \frac{n}{H}$, with $t$ ranging in an interval $(0, t_0]$.

\begin{definition} \label{old1.54}
  A type-I Einstein flow is a CMC Einstein flow
  for which there is some $C < \infty$ so that
  $|\Rm|_T \le C t^{-2}$ for all $t \in (0, t_0]$.
\end{definition}

\begin{example} \label{2.55}
  The Einstein flows in Example \ref{2.9} are all type-I.

  Consider a locally homogeneous Einstein flow with a crushing singularity.
  As the spacetime Ricci tensor
  vanishes, the curvature tensor is determined by the spacetime
  Weyl curvature.
  When $\dim(X) = 3$, the Weyl curvature is expressed in terms of
  ``electric'' and ''magnetic''
  tensors \cite[Section 1.1.3]{Ellis-Wainwright (1997)}.
  After normalization by the Hubble time, the tensor components
  can be written as
  polynomials in the Wainwright-Hsu variables
  $\Sigma_+, \Sigma_-, N_+, N_-, N_1$ \cite[(6.37)]{Ellis-Wainwright (1997)}.
    Hence the Einstein flow will be type-I provided that these
    variables remain bounded as one approaches the singularity.
    From
    \cite{Ringstrom (2001)}, this is the case for homogeneous Einstein flows of
    Bianchi type IX,
    i.e. flows of left-invariant data on $\SU(2)$
    \cite[Section 6.4]{Ellis-Wainwright (1997)}, some of which exhibit
    Mixmaster behavior.
  \end{example}

Let $B_{h(t)}(x,t)$ denote the time-$t$ metric ball of radius $t$ around $x$.

\begin{definition} \label{2.56}
  If ${\mathcal E}$ is a CMC Einstein flow then a sequence
  $\{(x_i, t_i)\}_{i=1}^\infty$ in $X \times (0, t_0]$ is noncollapsing
    if $\vol \left( B_{h(t_i)}(x_i, t_i)  \right) \ge v_0 t_i^n$ for all
    $i$, and some $v_0 > 0$.  We say that ${\mathcal E}$ is noncollapsed
    if there is some $v_0 > 0$ so that for all $(x,t) \in
    X \times (0, t_0]$, we have
      $\vol \left( B_{h(t)}(x, t)  \right) \ge v_0 t^n$.
\end{definition}

Recall the rescaling from (\ref{old1.36}).
We write the rescaled Einstein flow as ${\mathcal E}_s$. It is also type-I,
with the same constant $C$.

\begin{proposition} \label{old1.55}
  Let ${\mathcal E}$ be a type-I
  Einstein flow on an $n$-dimensional manifold $X$.
  Suppose that it is
  defined on a time-interval $(0, t_0]$ and has complete time slices.
    Let
          $\{(x_i, t_i)\}_{i=1}^\infty$ be a noncollapsing sequence in
      $X \times (0, t_0]$ with $\lim_{i \rightarrow \infty} t_i = 0$.
      Then after passing to a subsequence, which we relabel as
      $\{t_i\}_{i=1}^\infty$ and $\{x_i\}_{i=1}^\infty$, there is a limit
      $\lim_{i \rightarrow \infty} {\mathcal E}_{t_i} =
      {\mathcal E}^\infty$
      in the pointed weak $W^{2,p}$-topology and the pointed
      $C^{1,\alpha}$-topology. The limit flow
      ${\mathcal E}^\infty$ is
      defined on the time interval $(0, \infty)$. Its time slices
      $\{(X^\infty, h^\infty(u))\}_{u > 0}$ are complete. Its
      lapse function $L^\infty$ is uniformly bounded below by a
      positive constant.
\end{proposition}

The proof of Proposition \ref{old1.55} is essentially the same as that of \cite[Corollary 2.54]{Lott (2018)}, which is based on
\cite{Anderson (2001)}. It relies on the fact that the curvature bound,
along with the noncollapsing, implies uniform bounds on the local geometry
in the pointed $W^{2.p}$-topology or the pointed $C^{1,\alpha}$-topology 
\cite{Anderson (2003),Chen-LeFloch (2009)}.

\begin{example} \label{2.58}
  For Example \ref{2.9}(1), ${\mathcal E}^\infty = {\mathcal E}$. For
  Example \ref{2.9}(2), ${\mathcal E}^\infty$ is the product of $\R^{n^\prime}$
  with the
  Lorentzian cone on the $(n-n^\prime)$-dimensional Einstein manifold.
  For Example \ref{2.9}(3), ${\mathcal E}^\infty$ is a Kasner solution
  on $\R^3$ with
  $M = \diag \left( \frac23,\frac23,- \: \frac13 \right)$.
  For Example \ref{2.9}(4), ${\mathcal E}^\infty$ is a Kasner solution on
  $\R^n$ with the same matrix $M$ as the original flow. 
\end{example}

\begin{example} \label{2.59}
  Suppose that
  ${\mathcal E}$ is a Mixmaster flow of Bianchi type IX
  \cite[Section 6.4.3]{Ellis-Wainwright (1997)}.
  As mentioned in Example \ref{2.55}, it is a type-I Einstein flow.
  We don't know if it is necessarily noncollapsing, but let's suppose
  that it is noncollapsing.  We expect that any pointed rescaling limit
  ${\mathcal E}^\infty$ will be a Kasner solution or a
  Bianchi type II Taub solution \cite[Section 9.2.1]{Ellis-Wainwright (1997)}.
    {\it A priori}, the rescaling limit could also be a Mixmaster solution.
    However, numerical evidence indicates that the mixing slows down as
    $t \rightarrow 0$; see \cite[Figure 12]{Creighton-Hobill (1994)},
    which 
    shows the evolution as a function of $\log (- \log t)$.
    (The authors of \cite{Creighton-Hobill (1994)} inform me that the
    vertical axis of Figure 12 should be labelled by $\log N$ instead of
    $Z$.)
  \end{example}

\subsection{Diameter bounds} \label{subsect3.3}

Let ${\mathcal S}$ be the collection of
Einstein flows that generate Lorentzian cones
over compact $n$-dimensional
Riemannian Einstein manifolds with Einstein constant $-(n-1)$.
They are defined on the time interval $(0, \infty)$.

\begin{proposition} \label{2.60}
  Suppose that a type-I Einstein flow ${\mathcal E}$ has
  $\liminf_{t \rightarrow 0} t^{-1} \diam(X, h(t)) < \infty$.
  Then there is a sequence $\{t_i\}_{i=1}^\infty$ with
  $\lim_{i \rightarrow \infty} t_i = 0$ so that as $i \rightarrow
  \infty$, the
  rescaled Einstein flows ${\mathcal E}_{t_i}$ approach ${\mathcal S}$
  in the weak $W^{2,p}$-topology and the $C^{1, \alpha}$-topology.
  \end{proposition}
\begin{proof}
  Choose a sequence $\{t_i\}_{i=1}^\infty$ with
  $\lim_{i \rightarrow \infty} t_i = 0$ and
  $t_i^{-1} \diam(X, h(t_i)) < D$, for all $i$ and some $D < \infty$.
  As $t^{-n} \vol(X, h(t))$ is monotonically nonincreasing in $t$, it is
  uniformly bounded below on $(0,t_0]$
    by some positive constant. Let $\{x_i\}_{i=1}^\infty$
    be a sequence of points in $X$. Then by the Bishop-Gromov
    inequality, the sequence of points
    $\{(x_i, t_i)\}_{i=1}^\infty$ is noncollapsing for
    ${\mathcal E}$.
      After passing to a subsequence, we can extract a rescaling limit
  $\lim_{i \rightarrow \infty} {\mathcal E}_{t_i} =
  {\mathcal E}^\infty$ on a manifold $X^\infty$ that is compact, since
  $D < \infty$. In particular, $X^\infty$ is diffeomorphic to $X$.

  From the monotonicity of the normalized volume  of ${\mathcal E}$,
  it follows that
  $u^{-n} \vol(X^\infty, h^\infty(u))$ is independent of
  $u \in (0, \infty)$. The claim now follows from Proposition \ref{old1.15},
  whose proof works for $W^{2,p}$-regular metrics.
  \end{proof}

\begin{corollary} \label{2.61}
  If $\dim(X) = 3$ then under the hypotheses of Proposition \ref{2.60}, the
  original flow ${\mathcal E}$ is a
  Lorentzian cone over 
  a compact Riemannian $3$-manifold with constant sectional curvature $-1$.
  \end{corollary}
\begin{proof}
  By Proposition \ref{2.60}, there is a hyperbolic metric $h^\infty(1)$ on $X$,
  unique up to isometry by Mostow rigidity.  From the constraint equation
  (\ref{old1.2}), the scalar curvature $R(t_0)$ of $h(t_0)$ satisfies
  $R(t_0) \: \ge \: - \: \frac{6}{t_0^2}$. Then from Perelman's work
  \cite[Section 93.4]{Kleiner-Lott (2008)},
\begin{equation} \label{2.62}
  t_0^{-3} \vol(X, h(t_0)) \ge \vol(X, h^\infty(1)).
\end{equation}
From the existence of the
limiting flow in the proof of Proposition \ref{2.60},
\begin{equation} \label{2.63}
  \lim_{i \rightarrow \infty}
  t_i^{-3} \vol(X, h(t_i)) = \vol(X, h^\infty(1)).
  \end{equation}
Since $t^{-3} \vol(X, h(t))$ is nonincreasing in $t$, it follows that
$t^{-3} \vol(X, h(t)) = \vol(X, h^\infty(1))$ for all $t \in (0, t_0]$.
  The claim now follows from Proposition \ref{old1.15}.
\end{proof}

\subsection{Causal pasts} \label{subsect3.4}

Given $x \in X$ and $t \in [0, t_0)$, let $J_-(x,t)$ denote the
  causal past of $(x,t)$, i.e. the spacetime points that can be reached from
  past-directed timelike or null curves starting from $(x,t)$. 
  The next result is fairly standard but we include it for completeness.
  
\begin{proposition} \label{2.64}
  Let ${\mathcal E}$ be a CMC Einstein flow, defined for Hubble time
  $t \in (0, t_0]$. Suppose that there is some continuous function
    $f : (0, t_0] \rightarrow \R^+$, with
      $\int_0^{t_0} \frac{dt}{f(t)} < \infty$,
so that
$h(t) \ge f^{2}(t) h(t_0)$ for all $t \in (0, t_0]$.
    Then for any $x^\prime \in X$, the causal past $J_-(x^\prime,t)$ satisfies
      $\lim_{t \rightarrow 0} \diam(J_-(x^\prime,t), h(t_0)) = 0$.
  \end{proposition}
\begin{proof}
  A past causal curve $\gamma(s) = (x(s), s)$ satisfies
  $- L^2 + h_s \left( \frac{dx}{ds}, \frac{dx}{ds} \right) \le 0$.
  By (\ref{old1.14}), we have
  $h_s \left( \frac{dx}{ds}, \frac{dx}{ds} \right) \le L^2 \le 1$.
  The length of $\gamma$ with respect to
    $h(t_0)$ satisfies
    \begin{equation} \label{2.65}
      {\mathcal L} = \int_0^{t} \sqrt{h_{t_0}
      \left( \frac{dx}{ds}, \frac{dx}{ds} \right)} \: ds =
      \int_0^{t} \sqrt{
        \frac{h_{t_0}
        \left( \frac{dx}{ds}, \frac{dx}{ds} \right)}{h_{s}
        \left( \frac{dx}{ds}, \frac{dx}{ds} \right)}
      }
      \: \sqrt{h_{s}
      \left( \frac{dx}{ds}, \frac{dx}{ds} \right)} \: ds \le
      \int_0^{t} \frac{ds}{f(s)}.
    \end{equation}
Hence $\diam(J_-(x^\prime,t), h(t_0)) \le \int_0^{t} \frac{ds}{f(s)}$. 
    The proposition follows.
  \end{proof}

\begin{example} \label{2.66}
  The Kasner solution of Example \ref{2.9}(4) satisfies the hypotheses of
  Proposition \ref{2.64} provided that the eigenvalues of $M$ are
  strictly less than one.
  \end{example}

Under the assumptions of Proposition \ref{2.64}, for any distinct
$x, x^\prime \in X$, if $t$ is
small enough then $J_-(x,t)$ and  $J_-(x^\prime,t)$ are disjoint.
This is not true for general CMC Einstein flows.  However, we show that
one can often find many points whose causal pasts are disjoint
on a relatively long time interval.

\begin{proposition} \label{2.67}
  Let ${\mathcal E}$ be a noncollapsed type-I CMC Einstein flow with
  $\lim_{t \rightarrow 0} t^{-1} \diam(X, h(t)) = \infty$. Given
  $N \in \Z^+$, $\Lambda > 1$ and $x^\prime \in X$, there is some $\widehat{t} \in
  (0, t_0]$ with the following property.  Given $t \in (0, \widehat{t}]$,
      there are $N$ points $\{x_j\}_{=1}^N$ in $X$, with $x_1 = x^\prime$,
      so that if $j \neq
      j^\prime$ then the causal pasts $J^-(x_j,t)$ and $J^-(x_{j^\prime},t)$
      are disjoint on the time interval $[\Lambda^{-1} t, t]$.
\end{proposition}
\begin{proof}
  Given ${\mathcal E}$, suppose that the proposition fails. Then for some
  $N \in \Z^+$ and $\Lambda > 1$, there is a sequence of times
  $\{t_i\}_{i=1}^\infty$ with $\lim_{i \rightarrow \infty} t_i = 0$ so that
  the proposition fails for $t = t_i$.
  After passing to a subsequence, we
  can take a pointed rescaling limit $\lim_{i \rightarrow \infty} {\mathcal E}_{t_i} = {\mathcal E}^\infty$, that exists for times
  $u \in (0, \infty)$. By the diameter assumption, $X^\infty$ is noncompact.
  Because of the uniformly bounded $C^{1,\alpha}$-geometry of
  ${\mathcal E}^\infty$ on the time
  interval $[\Lambda^{-1}, 1]$
  (see the comments after Proposition \ref{old1.55}),
  there is some $R < \infty$ so that if
  $p,p^\prime \in X^\infty$ have $d_{h^\infty(1)}(p, p^\prime) \ge R$ then
    $J^-(p,1)$ and $J^-(p^\prime,1)$ are disjoint on the time interval
  $[\Lambda^{-1}, 1]$. Choose points $\{p_j\}_{j=1}^N$ in $X^\infty$,
  with $p_1 = x^\infty$, so that
    $d_{h^\infty(1)}(p_j, p_{j^\prime}) \ge 2R$ for $j \neq j^\prime$.
      For large $i$, let $\{x_{i,j}\}_{j=1}^N$ be points in
      $(X, t_i^{-2} h(t_i))$
      that are Gromov-Hausdorff approximants to the points $\{p_j\}_{j=1}^N$ in
      $(X_\infty, h^\infty(1))$, with $x_{i,1} = x^\prime$.
      From the $C^{1,\alpha}$-convergence
      when taking the rescaling limit, we conclude that for large $i$,
      if $j \neq
      j^\prime$ then the causal pasts $J^-(x_{i,j},t_i)$ and
      $J^-(x_{i,j^\prime},t_i)$
      are disjoint on the time interval $[\Lambda^{-1} t_i, t_i]$.
      This is a contradiction.
\end{proof}

In the three dimensional case, we can strengthen the conclusion of
Proposition \ref{2.67}.

\begin{corollary} \label{2.68}
  Let ${\mathcal E}$ be a noncollapsed type-I CMC Einstein flow with
  $\dim(X) = 3$. Then either
  \begin{enumerate}
      \item ${\mathcal E}$ is a
  Lorentzian cone over 
  a compact Riemannian $3$-manifold with constant sectional curvature $-1$, or
  \item Given
  $N \in \Z^+$, $\Lambda > 1$ and $x^\prime \in X$, there is some $\widehat{t} \in
  (0, t_0]$ with the following property.  Given $t \in (0, \widehat{t}]$,
      there are $N$ points $\{x_j\}_{j=1}^N$ in $X$, with $x_1 = x^\prime$, so that if $j \neq
      j^\prime$ then the causal pasts $J^-(x_j,t)$ and $J^-(x_{j^\prime},t)$
      are disjoint on the time interval $[\Lambda^{-1} t, t]$.
      \end{enumerate}
\end{corollary}
\begin{proof}
  This follows from Corollary \ref{2.61} and Proposition \ref{2.67}.
  \end{proof}

\begin{example} \label{2.69}
  \begin{enumerate}
  \item If ${\mathcal E}$ is a Lorentzian cone over a compact
    hyperbolic $3$-manifold $X$ then there is some $\Lambda > 0$ so
    that for any $x \in X$ and any $t \in (0, t_0]$, the intersection
    of $J^-(x,t)$ with the slice at time $\Lambda^{-1} t$ is all of
    $X$. This shows that the two cases in the conclusion of Corollary
    \ref{2.68} are distinct.
  \item Suppose that ${\mathcal E}$ is a Kasner solution as in
    Example \ref{2.66}.  Using the spatial homogeneity, we can strengthen
    the conclusion of Corollary \ref{2.68} to say that
    $J^-(x_j, t)$ and $J^-(x_{j^\prime}, t)$ are disjoint on the time
    interval $(0, t]$. As $N \rightarrow \infty$, we can assume that
      the points
      $\{x_j\}_{j=1}^N$ become uniformly distributed on $X$.
  \item Suppose that ${\mathcal E}$ is the product of
    $T^2$ with the Lorentzian cone over a circle.  Given a point
    $x = (x_{T^2}, x_{S^1}) \in X$, we can take the points $\{x_j\}_{j=1}^N$
    to lie on $T^2 \times \{x_{S^1}\}$ and we can strengthen
    the conclusion of Corollary \ref{2.68} to say that
    $J^-(x_j, t)$ and $J^-(x_{j^\prime}, t)$ are disjoint on the time
    interval $(0, t]$. 
  \item Let $\widetilde{\mathcal E}$ be the product of $\R^2$ with the
    Lorentzian cone over $\R$. Let $\Gamma$ be a lattice in $\R^3$ with
    irrational entries.  Let ${\mathcal E}$ be the
    $\Gamma$-quotient of $\widetilde{\mathcal E}$, an Einstein flow on
    $T^3$. We cannot strengthen
    the conclusion of Corollary \ref{2.68} to say that
    $J^-(x_j, t)$ and $J^-(x_{j^\prime}, t)$ are disjoint on the time
    interval $(0, t]$.
    \end{enumerate}
  \end{example}

We now localize Proposition \ref{2.67}
to an arbitrary open subset $U$ of $X$.

\begin{proposition} \label{2.70}
  Let ${\mathcal E}$ be a noncollapsed type-I CMC Einstein flow.
  Given $N \in \Z^+$, $\Lambda > 1$, $\epsilon > 0$, $\alpha \in (0,1)$,
  an open set $U \subset X$ and a point $x^\prime \in U$, there is
  some $\widehat{t} \in (0, t_0]$ with the following property. For $t \in (0, \widehat{t}]$,
    \begin{enumerate}
    \item     The rescaled pointed flow ${\mathcal E}_{t}$ on $(X, x^\prime)$ is $\epsilon$-close in the
      pointed $C^{1,\alpha}$-topology to a Lorentzian cone 
      over a Riemannian Einstein metric with Einstein constant $-(n-1)$, having $U$
      as a bounded subset of the approximation region, or
 \item There are $N$ points $\{x_j\}_{=1}^N$ in $U$, with $x_1 = x^\prime$, so that if $j \neq
      j^\prime$ then the causal pasts $J^-(x_j,t)$ and $J^-(x_{j^\prime},t)$
      are disjoint on the time interval $[\Lambda^{-1} t, t]$.      
      \end{enumerate}
  \end{proposition}
\begin{proof}
    Given ${\mathcal E}$, suppose that the proposition fails. Then for some
    $N \in \Z^+$, $\Lambda > 1$, $\epsilon > 0$,
    $\alpha \in (0,1)$, $U \subset X$ and $x^\prime \in U$, there is a sequence of times
  $\{t_i\}_{i=1}^\infty$ with $\lim_{i \rightarrow \infty} t_i = 0$ so that
    the proposition fails for $t = t_i$.  In particular, for each $i$, the rescaled pointed flow
    ${\mathcal E}_{t_i}$ is not $\epsilon$-close to a Lorentzian cone as described in case (1).

    Suppose first that $\liminf_{i \rightarrow \infty} t_i^{-1} \diam(U, h(t_i)) < \infty$. Choose $D < \infty$ so that
    after passing to a subsequence, we have $t_i^{-1} \diam(U, h(t_i)) \le D$ for all $i$.
With the basepoint $x^\prime$, after passing to a subsequence, we
  can take a pointed rescaling limit $\lim_{i \rightarrow \infty} {\mathcal E}_{t_i} = {\mathcal E}^\infty$, that exists for times
  $u \in (0, \infty)$. For large $i$, the Gromov-Hausdorff approximants of $(U, t_i^{-2} h(t_i))$ lie in
  $B(x^\infty, 2D) \subset (X^\infty, h^\infty(1))$. Hence there is a uniform upper bound
  $t_i^{-n} \vol(U, h(t_i)) \le V < \infty$. Now $t^{-n} \dvol_{h(t)}$ is pointwise nonincreasing in $t$, on $U$.
  From the monotone convergence theorem, there is some $\widetilde{x}\in U$ so that
  $\lim_{t \rightarrow 0} t^{-n} \dvol_{h(t)}(\widetilde{x}) < \infty$.
  It follows from the strong maximum principle, as in
  \cite[Proposition 3.5]{Lott (2018)}, that ${\mathcal E}^\infty$ is a
  Lorentzian cone over a Riemannian Einstein metric on $U$ with Einstein constant $-(n-1)$. Then for large $i$,
  the rescaled pointed flow
  ${\mathcal E}_{t_i}$ is $\epsilon$-close to the Lorentzian cone, which is a contradiction. 

  Hence $\lim_{i \rightarrow 0} t_i^{-1} \diam(U, h(t_i)) = \infty$. We can now apply the argument in the
  proof of Proposition \ref{2.67}.
  Hence for large $i$, case (2) is satisfied.  This is a contradiction.
\end{proof}

We now prove a measure-theoretic version.

\begin{proposition} \label{2.71}
  Let ${\mathcal E}$ be a noncollapsed type-I CMC Einstein flow.
  Given $\Lambda > 1$, $\alpha \in (0,1)$ and $\epsilon > 0$, there are
  some $\widehat{t} \in (0, t_0]$ and $\Lambda^\prime < \infty$ with the following property.
    Choose $t \in (0, \widehat{t}]$ and $x \in X$. Let $V$ be the set of points $y \in X$
such that
      the causal pasts $J^-(x,t)$ and $J^-(y,t)$
      intersect on the time interval $[\Lambda^{-1} t, t]$.
      Then
      \begin{enumerate}
    \item
      $\vol(V, h(t_0)) \le \epsilon t_0^n$, or
    \item
      There are some $t^\prime \in [t,  \Lambda^\prime t]$ and $y \in V$ so that
      the rescaled pointed flow ${\mathcal E}_{t^\prime}$ on $(X, y)$ is $\epsilon$-close
         in the
      pointed $C^{1,\alpha}$-topology to a Lorentzian cone 
      over a Riemannian Einstein metric with Einstein constant $-(n-1)$.
            \end{enumerate}
  \end{proposition}
\begin{proof}
    Given ${\mathcal E}$, suppose that the proposition fails. Then for some
    $\Lambda > 1$, $\alpha \in (0,1)$ and $\epsilon > 0$, there is a sequence of
    points $\{x_i\}_{i=1}^\infty$ and times
  $\{t_i\}_{i=1}^\infty$  with $t_i \le i^{-1} t_0$ so that
    the proposition fails if we take $x = x_i$, $t = t_i$ and $\Lambda^\prime = i$.
    That is, for all $i$,
    \begin{enumerate}
    \item
      $\vol(V_i, h(t_0)) > \epsilon t_0^n$, and
    \item
      There are no $t^\prime \in [t_i,  i t_i]$ and $y \in V_i$ so that
      the rescaled pointed flow ${\mathcal E}_{t^\prime}$ on $(X, y)$ is $\epsilon$-close
         in the
      pointed $C^{1,\alpha}$-topology to a Lorentzian cone 
      over a Riemannian Einstein metric with Einstein constant $-(n-1)$.
            \end{enumerate}

    From the uniformly bounded geometry
(see the comments after Proposition \ref{old1.55}),
    there is some
    $D < \infty$ so that
    for all $i$, the corresponding subset $V_i$ lies in
    $B_{h(t_i)}(x_i, Dt_i)$. In particular, there is some
    ${\mathcal V} < \infty$ so that
    for all $i$, we have
$t_i^{-n} \vol(V_i, h(t_i)) \le {\mathcal V}$.

    Now
    \begin{equation} \label{2.72}
      \frac{
        \int_{V_i} \frac{t_0^{-n} \dvol_{h(t_0)} }{ t_i^{-n} \dvol_{h(t_i)} }
        \dvol_{h(t_i)}
        }{
          \int_{V_i} \dvol_{h(t_i)}
        } = 
        \frac{t_0^{-n} \vol(V_i, h(t_0))}{t_i^{-n} \vol(V_i, h(t_i))} 
       \ge 
\frac{\epsilon}{\mathcal V}.      
    \end{equation}
    Thus there is some $y_i \in V_i$ with
    \begin{equation} \label{2.73}
      \frac{t_0^{-n} \dvol_{h(t_0)}(y_i) }{ t_i^{-n} \dvol_{h(t_i)}(y_i) }
\ge 
\frac{\epsilon}{\mathcal V}.      
    \end{equation}
    From the monotonicity of the normalized volume, we also have
\begin{equation} \label{neweqn}
      \frac{(ut_i)^{-n} \dvol_{h(ut_i)}(y_i) }{ t_i^{-n} \dvol_{h(t_i)}(y_i) }
\ge 
\frac{\epsilon}{\mathcal V}      
\end{equation}
for all $u \in \left[ 1, \frac{t_0}{t_i} \right]$.

    After passing to a subsequence, we can assume that the
    ${\mathcal E}_{t_i}$'s, on the pointed spaces $(X, y_i)$, have a pointed limit
    ${\mathcal E}^\infty$.
The limit of (\ref{neweqn}) gives
    \begin{equation} \label{2.74}
        \frac{
          u^{-n} \dvol_{{h}^\infty(u)}({y}^\infty)
        }{
\dvol_{{h}^\infty(1)}({y}^\infty) }
\ge 
\frac{\epsilon}{\mathcal V}
    \end{equation}
    for all $u \ge 1$.
      From \cite[Proposition 3.5]{Lott (2018)},
      there is some $\Lambda^\prime < \infty$ so that
      ${\mathcal E}^\infty_{\Lambda^\prime}$ on $({X}^\infty, {y}^\infty)$
      is $\frac12 \epsilon$-close in the pointed $C^{1,\alpha}$-topology to a Lorentzian cone as in
      case (2). Then for large $i$, 
      ${\mathcal E}_{\Lambda^\prime t_i}$ on $({X}, y_i)$ is
      $\epsilon$-close in the pointed $C^{1,\alpha}$-topology to such a Lorentzian cone.
      This is a contradiction.
      \end{proof}

\subsection{Kasner-type limits} \label{subsect3.5}

In Theorem \ref{0.13} we showed that if $R \le 0$ and $\dvol_0 \neq 0$ then
as $t \rightarrow 0$, there is Kasner-like behavior in an integral sense.
We now improve this to a pointwise statement under the additional assumption
that the flow is noncollapsed and type-I.

Let ${\mathcal S}$ be the collection of pointed
Einstein flows with $L = \frac{1}{n}$, $R = 0$ and
$|K|^2 = H^2$, defined on the time interval $(0, \infty)$; c.f. (\ref{2.14}).

\begin{proposition} \label{2.75}
  Let ${\mathcal E}$ be a type-I CMC Einstein flow.
  Suppose that $R \le 0$. Let $x \in X$ be such that
  $\dvol_0(x) = \lim_{t \rightarrow 0} t^{-1} \dvol_t(x) \neq 0$.
  Suppose that the flow is noncollapsing at $x$ as $t \rightarrow 0$, i.e. 
  that there is a uniform lower bound
  $\vol(B_{h(t)}(x,t)) \ge v_0 t^n$ for some $v_0 > 0$.
  Then as $s \rightarrow 0$,
  the rescaled Einstein flows ${\mathcal E}_s$, pointed at $x$,
  approach ${\mathcal S}$ in the pointed weak $W^{2,p}$-topology and the
  pointed $C^{1, \alpha}$-topology.
\end{proposition}
\begin{proof}
  From Proposition \ref{2.30} and Corollary \ref{2.31},
  we know that $L \ge \frac{1}{n}$, and $t^{-1} \dvol_{h(t)}(x)$ is
  monotonically nondecreasing in $t$.
  Let $\{s_i\}_{i = 1}^\infty$ be a sequence converging to zero.
  After passing to a subsequence, we can assume that there is a limit
  $\lim_{i \rightarrow \infty} {\mathcal E}_{s_i} = {\mathcal E}^\infty$
  in the pointed weak $W^{2,p}$-topology and the
  pointed $C^{1, \alpha}$-topology. Equation (\ref{2.20}) implies that
  \begin{equation} \label{2.76}
   n \int_0^{t_0} \left( L - \frac{1}{n} \right) \frac{dt}{t} =
    \ln \frac{t_0^{-1} \dvol_{t_0}(x)}{\dvol_0(x)}
 < \infty.
    \end{equation}
  It follows that $L^\infty(x,u) = \frac{1}{n}$ for all $u \in (0, \infty)$.
  As in \cite[Proof of Proposition 3.5]{Lott (2018)}, the strong maximum
  principle implies that $L^\infty = \frac{1}{n}$ and $R = 0$. Then
  the constraint equation (\ref{old1.2}) gives $|K|^2 = H^2$.
\end{proof}

\begin{remark} \label{2.77}
  If we omit the noncollapsing assumption then the conclusion
  of Proposition \ref{2.75} still holds,
  provided that we take the limit flow to live on an \'etale groupoid; c.f.
  \cite[Proposition 3.5]{Lott (2018)}.
  \end{remark}

\begin{example} \label{2.78}
  Let ${\mathcal E}$ be a Bianchi-VIII NUT solution on a circle bundle
  over a higher genus surface
  \cite[Section 9.14]{Ellis-Wainwright (1997)}. Then the
  hypotheses of Proposition \ref{2.75} are satisfied.
  As $t \rightarrow 0$, the diameter of the quotient
  surface approaches a constant,
  while the diameter of the circle fiber goes like $t$. The pointed
  rescaling limit ${\mathcal E}^\infty$ is $\R^2$ times a Lorentzian
  cone over a circle.
  \end{example}

\begin{example} \label{2.79}
  As mentioned in Example \ref{2.29},  Mixmaster
  solutions of Bianchi type VIII have $R \le 0$. However, we
  expect that Proposition \ref{2.75} does not give additional
  information.  That is, we expect that
  $\dvol_0 = 0$. The reason is the infinite number of Taub Bianchi-II
  transitions
\cite{Brehm (2016),Liebscher-Harterich-Webster-Georgi (2011)},
which we expect will drive $t^{-1} \dvol_{h(t)}$ to zero
as $t \rightarrow 0$; see (\ref{2.20}).
  \end{example}

\section{Discussion}

In Section 2 of the paper we considered the initial behavior of
vacuum spacetimes with a toral symmetry.
We showed that Gowdy spacetimes in arbitrary dimension have AVTD-like
behavior for the $G$-component of the metric.  This complements
earlier results of Ringstrom \cite{Ringstrom (2006)}. We also considered
four dimensional nonGowdy spacetimes with a $T^2$-symmetry. We again showed
AVTD-like behavior under certain assumptions.  Since such spacetimes
probably do not have AVTD-like behavior in general, it would be interesting
to clarify the borderline between these two types of behavior.
One could also consider other symmetry classes.

Section 3 of the paper was devoted to vacuum spacetimes with a
CMC foliation. We obtained results about the initial geometry
using a normalized
volume functional, which is monotonically nondecreasing in time when the
spatial slices have nonpositive scalar curvature.  One question is
whether there are other relevant monotonic functionals.
We also considered CMC vacuum spacetimes
that are type-I and locally
noncollapsed, as one approaches an initial singularity.
One can ask how widely these assumptions
hold. When they do hold, we used rescaling and compactness arguments to
say something about causal pasts. It is possible that such
rescaling and compactness methods can be combined with
other techniques to obtain further results.

\end{document}